\documentclass[a4paper,11pt]{amsart}

\usepackage{epsfig}
\usepackage{amsthm,amsfonts}
\usepackage{amssymb,graphicx,color}

\graphicspath{ {Images/} }
\usepackage[all]{xy}
\usepackage{float}
\usepackage{cancel} %\cancel{}, \bcancel{}, \xcancel{}
\usepackage{verbatim}
\usepackage{hyperref}
\usepackage{enumitem}
\usepackage{tikz}
\usetikzlibrary{arrows}

\newtheorem{theorem}{Theorem}[section]
\newtheorem*{theorem*}{Theorem}
\newtheorem{lemma}[theorem]{Lemma}
\newtheorem{corollary}[theorem]{Corollary}
\newtheorem{proposition}[theorem]{Proposition}
\newtheorem{definition}[theorem]{Definition}
\newtheorem{example}[theorem]{Example}
\newtheorem{claim}{Claim}
\newtheorem{remark}[theorem]{Remark}

\newcommand{\R}{\mathbb{R}}
\newcommand{\C}{\mathbb{C}}

\newcommand{\N}{\mathbb{N}}

\begin{document}

\title[Multiplicity, regularity and blow-spherical equivalence]
{Multiplicity, regularity and blow-spherical equivalence of real analytic sets}
% \title[Blow-spherical equivalence and multiplicity]
% {Blow-spherical equivalence and multiplicity of real analytic sets}

\author[J. Edson Sampaio]{Jos\'e Edson Sampaio}

\address{Jos\'e Edson Sampaio:  Departamento de Matem\'atica, Universidade Federal do Cear\'a,
	      Rua Campus do Pici, s/n, Bloco 914, Pici, 60440-900, 
	      Fortaleza-CE, Brazil.  
          E-mail: {\tt edsonsampaio@mat.ufc.br}   
}

\thanks{The author was partially supported by CNPq-Brazil grant 303811/2018-8. % Agradecer ao pronex
% , by the ERCEA 615655 NMST Consolidator Grant and also by the Basque Government through the BERC 2018-2021 program and Gobierno Vasco Grant IT1094-16, by the Spanish Ministry of Science, Innovation and Universities: BCAM Severo Ochoa accreditation SEV-2017-0718. 
% The author was partially supported in the initial part of this paper by FUNCAP.
}

%\thanks{The authors were partially supported by CNPq-Brazil}

\begin{abstract}
This article is devoted to studying multiplicity and regularity of real analytic sets. We present an equivalence for real analytic sets, named blow-spherical equivalence, which generalizes differential equivalence and subanalytic bi-Lipschitz equivalence and, with this approach, we obtain several applications on analytic sets. On regularity, we show that blow-spherical regularity of real analytic implies $C^1$ smoothness only in the case of real analytic curves. On multiplicity, we present a generalization for Gau-Lipman's Theorem about differential invariance of the multiplicity in the complex and real cases, we show that the multiplicity ${\rm mod}\,2$ is invariant by blow-spherical homeomorphisms in the case of real analytic curves and surfaces and also for a class of real analytic foliations and is invariant by (image) arc-analytic blow-spherical homeomorphisms in the case of real analytic hypersurfaces, generalizing some results proved by G. Valette. We present also a complete classification of the germs of real analytic curves.
\end{abstract}

\maketitle

\section{Introduction}

% Recently, in \cite{BirbrairFG:2012}, \cite{Sampaio:2015}, \cite{BirbrairFG:2017} and \cite{Sampaio:2020} was approached the following equivalence relation.

Recently, L. Birbrair, A. Fernandes and V. Grandjean in \cite{BirbrairFG:2017} (see also \cite{BirbrairFG:2012} and \cite{Sampaio:2015}) defined a new equivalence, named blow-spherical equivalence, to study some properties of subanalytic sets such as, for example, to generalize thick-thin decomposition of normal complex surface singularity germs introduced in \cite{BirbrairNP:2014}. 
With the aim to study multiplicity and regularity as well as to present some classifications of real and complex analytic sets, we have the following weaker variation of the equivalence blow-spherical presented in \cite{BirbrairFG:2017}, namely also blow-spherical equivalence. 

\begin{definition}
Let $X\subset\R^n$ and $Y\subset\R^p$ be two subsets containing respectively the origin of $\R^n$ and $\R^p$.  A homeomorphism $\varphi\colon (X,0)\to (Y,0)$ is called a {\bf blow-spherical homeomorphism}, if the homeomorphism
$$
\beta_p^{-1}\circ \varphi\circ\beta_n\colon \beta_n^{-1}(X\setminus \{0\})\to \beta_p^{-1}(Y\setminus \{0\})
$$
extends as a continuous mapping $\varphi'\colon \overline{\beta_n^{-1}(X\setminus \{0\})}\to \overline{\beta_p^{-1}(Y\setminus \{0\})}$, where $\beta _k:\mathbb{S}^{k-1}\times [0,+\infty)\to \R^k$ is the mapping given by $\beta_k(x,r)=rx$. In this case, we say that the germs $(X,0)$ and $(Y,0)$ are said {\bf blow-spherical equivalent} or {\bf blow-spherical homeomorphic}.
\end{definition}

Roughly speaking, two subset germs of Euclidean spaces are called blow-spherical equivalent, if their spherical modifications are homeomorphic and, in particular, this homeomorphism induces a homeomorphism between their tangent links. In particular, this equivalence lives strictly between topological equivalence and subanalytic bi-Lipschitz equivalence and between topological equivalence and differential equivalence.

In \cite{Sampaio:2020}, the author presented some results on multiplicity,  regularity and blow-spherical geometry of complex analytic sets. In this article, we present similar results in the case of real analytic sets, for example, we obtain some results on real versions of Zariski's multiplicity conjecture, we have that blow-spherical regular real analytic curves are $C^1$ smooth and we obtain a complete classification to real analytic curves.

When the subject is about multiplicity, the more famous open problem is the Zariski's multiplicity conjecture. O. Zariski in 1971 (see \cite{Zariski:1971}) asked the  following.
\begin{enumerate}[leftmargin=*]\label{zariski}
\item[]{\bf Question A.} Let $f,g:(\C^n,0)\to (\C,0)$ be two reduced complex analytic functions. If there is a homeomorphism $\varphi:(\C^n,V(f),0)\to (\C^n,V(g),0)$, then is it $m(V(f),0)=m(V(g),0)$?
\end{enumerate} 
This is still an open problem. However, several authors approached it, as for example R. Ephraim in \cite{Ephraim:1976} and D. Trotman in \cite{Trotman:1977} showed that Question A has a positive answer when the homeomorphism $\varphi$ is a $C^1$ diffeomorphism. Another important result about this problem was proved by Y.-N. Gau and J. Lipman in the article \cite{Gau-Lipman:1983}, they proved that if $X,Y \subset\C^n$ are complex analytic sets and there is a homeomorphism $\varphi\colon (\C^n,X,0)\to(\C^n,Y,0)$ such that $\varphi$ and $\varphi^{-1}$ are differentiable at the origin, then the multiplicities of $X$ and $Y$ at the origin are equal (see \cite{Chirka:1989} for a definition of multiplicity).  
In the case of real analytic sets, Question A has a negative answer, as we can see in the following example.
\begin{example}\label{ex_zariski}
Let $X=\{(x,y)\in\R^2;\, y=0\}$, $Y=\{(x,y)\in\R^2;\, y^3=x^2\}$. $\varphi: \R^2\to \R^2$ given by $\varphi(x,y)=(x,x^{\frac{2}{3}}-y)$. Then, $\varphi$ is homeomorphism such that $\varphi(X)=Y$, but $m(X)\equiv 1\, {\rm mod\,}2$ and $m(Y)\equiv 0\, {\rm mod\,} 2$.
\end{example}

However, some authors approached some versions of Question A in the real case. For example, J.-J. Risler in \cite{Risler:2001} proved that multiplicity ${\rm mod\,} 2$ of a real analytic curve is invariant by bi-Lipschitz homeomorphisms. T. Fukui, K. Kurdyka and L. Paunescu in \cite{FukuiKP:2004} made the following conjecture
\begin{enumerate}[leftmargin=*]
\item[]{\bf Conjecture F-K-P.} Let $h : (\R^n, 0) \to (\R^n,0)$ be the germ of a subanalytic, arc-analytic, bi-Lipschitz homeomorphism, and let $X, Y\subset\R^n$ be two irreducible analytic germs. Suppose that $Y = h(X)$, then $m(X) = m(Y)$.
\end{enumerate} 
and proved that the multiplicity of a real analytic curve is invariant by arc-analytic bi-Lipschitz homeomorphisms. G. Valette in \cite{Valette:2010} proved that the multiplicity ${\rm mod\,} 2$ of a real analytic hypersurface is invariant by arc-analytic bi-Lipschitz homeomorphisms and the multiplicity ${\rm mod\,} 2$ of a real analytic surface is invariant by subanalytic bi-Lipschitz homeomorphisms. Recently, the author in \cite{Sampaio:2020c} proved a real version of Gau-Lipman's Theorem and in \cite{Sampaio:2020b} proposed the following conjecture
\begin{enumerate}[leftmargin=*]
\item[]{\bf Conjecture $A_{\R}$ (Lip).} Let $f,g:(\R^n,0)\to (\R,0)$ be two real analytic functions. If there is a bi-Lipschitz homeomorphism $\varphi\colon(\R^n,V(f),0)\to (\R^n,V(g),0)$, then $m(V(f),0)\equiv m(V(g),0)\, {\rm mod\,} 2$.
\end{enumerate}
and proved that it has a positive answer when $n=3$. 

In this article, we consider the following version of the Conjecture $A_{\R}$ (Lip).
\begin{enumerate}[leftmargin=*]
\item[]{\bf Conjecture $A_{\R}$ (BS).} Let $X,Y\subset \mathbb{R}^n$ be two real analytic sets. If there is a blow-spherical homeomorphism $\varphi\colon(\R^n,X,0)\to (\R^n,Y,0)$, then $m(X)\equiv m(Y)\, {\rm mod\,} 2$.
\end{enumerate}

Another subject approached in \cite{Sampaio:2020} was the study of the regularity of complex analytic sets, for example, it was proved there that any complex analytic set which is blow-spherical regular (see Definition \ref{def:bs_regular}) must be smooth. This is also a subject of interest for many mathematicians, for instance, Mumford in \cite{Mumford:1961} showed that a topological regular complex surface in $\C^3$,  with isolated singularity, is smooth. In high dimension, N. A'Campo in \cite{Acampo:1973} and L\^e D. T. in \cite{Le:1973} showed that if $X$ is a complex analytic hypersurface in $\C^n$ which is a topological submanifold, then $X$ is smooth. Recently, the author in \cite{Sampaio:2016} (see also \cite{BirbrairFLS:2016}) proved a version of Mumford's Theorem, he showed that if a complex analytic set is Lipschitz regular (see Definition \ref{def:lip_regular}) then it is smooth. % and in \cite{Sampaio:2020}, he showed that if a complex analytic set is blow-spherical regular at $0$ (see Definition \ref{def:bs_regular}) then it is smooth at $0$. 
Thus, since we are interested in the real case, for each positive integer $d$, we have the following question: % we show in this article that all the results above cited do not hold true, except Lipschitz regularity of real analytic curves.
\begin{enumerate}[leftmargin=0pt]
\item[]{\bf Question  BSR($d$)} Let $X\subset \mathbb R^n$ be a real analytic set with dimension $d$. Suppose that $X$ is blow-spherical regular at $0\in X$. Is it true that $X$ is $C^1$ smooth at $0$?
\end{enumerate}

In this article, we prove that above question has a positive answer if and only if $d=1$ (see Corollary \ref{regular} and Example \ref{naosuave_blow}).

On classification, let us remark that the bi-Lipschitz invariance of the multiplicity is an advance about a problem that has been extensively studied in the recent years: the classification of the real analytic surfaces under bi-Lipschitz homeomorphisms. 
Since any subanalytic bi-Lipschitz homeomorphism is a blow-spherical homeomorphism, we believe that the study of the blow-spherical equivalence can help in the bi-Lipschitz classification problem. It was presented in \cite{BirbrairFG:2017} some results on classification of the real analytic surfaces under blow-spherical homeomorphisms. However, the classification of the real analytic curves under blow-spherical homeomorphisms is still not known and this is why we present such a classification here.

Let us describe how this article is organized. 

In Section \ref{preliminaries}, we present the main preliminaries used in this article.

In Section \ref{bs_equiv}, we present the blow-spherical equivalence. We present some examples to emphasize that blow-spherical equivalence is different from some other well-known equivalences studied in Singularity Theory (see Subsection \ref{new_equiv}) and present some properties of that equivalence. 

In Section \ref{section:regularity}, we show that blow-spherical regular analytic curves are $C^1$ (see Corollary \ref{regular}) and we give examples that this does not hold true in other dimensions (see Example \ref{naosuave_blow}). We present also a complete classification of the real analytic curves (see Theorems \ref{curve_classif} and \ref{curve_classif_two}). 

In Section \ref{section:multiplicity}, we prove several results on invariance of the multiplicity. For instance, in Subsection \ref{subsec:blow-spherical-diff}, we prove the blow-spherical differential invariance of the multiplicity ${\rm mod\,} 2$ of real analytic sets (see Theorem \ref{blow-spherical-diff}), which is a generalization of the main result in \cite{Sampaio:2020c}. It is also presented a generalization of Gau-Lipman's Theorem in \cite{Gau-Lipman:1983} (see Theorem \ref{blow-spherical-diff-complex}). In Subsection \ref{subsec:curves}, we prove that the multiplicity ${\rm mod\,} 2$ of real analytic curves is invariant by blow-spherical homeomorphisms (see Proposition \ref{curve}) and as a consequence, we obtain, in Subsection \ref{subsec:foliations}, a result of invariance of multiplicity of real analytic foliations in $\R^2$ (see Corollary \ref{risler_result}). In Subsection \ref{subsec:surfaces}, we prove that the multiplicity ${\rm mod\,} 2$ of real analytic surfaces is invariant by embedded blow-spherical homeomorphisms (see Theorem \ref{surface}). In Subsection \ref{subsec:surfaces-non-embedded}, we prove that the multiplicity ${\rm mod\,} 2$ of real analytic surfaces is invariant by a (not necessary embedded) blow-spherical homeomorphism with an additional hypothesis on symmetry of its spherical blowing up (see Theorem \ref{surface-non-embedded}). Finally, in Subsection \ref{subsec:arc-analytic}, we prove the invariance of the multiplicity ${\rm mod\,} 2$ by (image) arc-analytic blow-spherical homeomorphism (see Theorem \ref{image-arc-Lip}). These results about multiplicity give generalizations of the Valette's results above said (see Remarks \ref{rem:gen_valette_one} and \ref{rem:gen_valette_two}).

\bigskip

% \noindent{\bf Acknowledgments}. The author wishes to thank my thesis advisors Alexandre Fernandes and Lev Birbrair by incentive and interest in this research. Some results of the article are part of the author's PhD thesis at Universidade Federal do Cear\'a. 

\section{Preliminaries}\label{preliminaries}
Here, all real analytic sets are supposed to be pure dimensional.
\begin{definition}
Let $A\subset \R^n$ be a subset such that $x_0\in \overline{A}$.
We say that $v\in \R^n$ is a tangent vector of $A$ at $x_0\in\R^n$ if there is a sequence of points $\{x_i\}\subset A\setminus \{x_0\}$ tending to $x_0\in \R^n$ and there is a sequence of positive numbers $\{t_i\}\subset\R^+$ such that 
$$\lim\limits_{i\to \infty} \frac{1}{t_i}(x_i-x_0)= v.$$
Let $C(A,x_0)$ denote the set of all tangent vectors of $A$ at $x_0\in \R^n$. We call $C(A,x_0)$ the {\bf tangent cone} of $A$ at $x_0$.
\end{definition}
\begin{remark}{\rm It follows from Curve Selection Lemma for subanalytic sets that, if $A\subset \R^n$ is a subanalytic set and $x_0\in \overline{A}$ is a non-isolated point, then the following holds true }
\begin{eqnarray*}
C(A,x_0)=\{v;\, \exists\, \alpha:[0,\varepsilon )\to \R^n\,\, \mbox{s.t.}\,\, \alpha(0)=x_0,\, \alpha((0,\varepsilon ))\subset A\,\, \mbox{and}\,\,\\ 
\alpha(t)-x_0=tv+o(t)\}. 
\end{eqnarray*}
\end{remark}
\begin{definition}
The mapping $\beta _n:\mathbb{S}^{n-1}\times \R^+\to \R^n$ given by $\beta_n(x,r)=rx$ is called {\bf spherical blowing-up} (at the origin) of $\R^n$.
\end{definition}
Note that $\beta _n:\mathbb{S}^{n-1}\times (0,+\infty )\to \R^n\setminus \{0\}$ is homeomorphism with inverse $\beta_n^{-1}:\R^n\setminus \{0\}\to \mathbb{S}^{n-1}\times (0,+\infty )$ given by $\beta_n^{-1}(x)=(\frac{x}{\|x\|},\|x\|)$.
\begin{definition}
The {\bf strict transform} of the subset $X$ under the spherical blowing-up $\beta_n$ is $X':=\overline{\beta_n^{-1}(X\setminus \{0\})}$ and the {\bf boundary} $\partial X'$ of {\bf strict transform} is $\partial X':=X'\cap (\mathbb{S}^{n-1}\times \{0\})$.
\end{definition}
Remark that $\partial X'=C_X\times \{0\}$, where $C_X=C(X,0)\cap \mathbb{S}^{n-1}$.

\subsection{Multiplicity and relative multiplicities}\label{subsection:mult_rel}

Let $X\subset \R^{n}$ be a $d$-dimensional real analytic set with $0\in X$ and 
$$
X_{\C}= V(\mathcal{I}_{\R}(X,0)),
$$
where $\mathcal{I}_{\R}(X,0)$ is the ideal in $\mathbb{C}\{z_1,...,z_n\}$ generated by the complexifications of all germs of real analytic functions that vanish on the germ $(X,0)$. We have that $X_{\C}$ is a germ of a complex analytic set and $\dim_{\C}X_{\C}=\dim_{\R}X$ (see Whitney \cite{Whitney:1957}, p. 546, Theorem 1 and p. 552, Lemma 8 and Lemma 9). Then, for each $(n-d)$-dimensional complex linear subspace $L\subset \C^n$ such that $L\cap C(X_{\C},0) =\{0\}$, there exists an open neighborhood $U\subset \C^n$ of $0$ and a complex analytic subset $\sigma \subset U'=\pi_L(U)$ such that $\# (\pi_L^{-1}(x)\cap X_{\C}\cap U)$ is constant for all point $x\in U'\setminus \sigma$, where $\pi_L\colon\C^{n}\to L^{\perp}$ is the orthogonal projection onto $L^{\perp}$ and $\dim \sigma <d$. The number $\# (\pi_L^{-1}(x)\cap X_{\C}\cap U)$ is the multiplicity of $X_{\C}$ at the origin and it is denoted by $m(X_{\C},0)$. 

\begin{definition}
With the above notation, we define the (real) multiplicity of $X$ at the origin by $m(X):=m(X_{\C},0)$.
\end{definition}

\begin{remark}
For $\pi:=\pi_L|_{\R^n}\colon \R^n\to \pi_L(\R^n)$, we have $m(X)\equiv \# (\pi^{-1}(x)\cap X\cap U) \,{\rm mod\,} 2$, for any $x\in \pi_L(\R^n)\setminus \sigma$.
\end{remark}

\begin{definition}
Let $X\subset \R^n$ be a subanalytic set such that $0\in \overline X$ is a non-isolated point. We say that $x\in\partial X'$ is {\bf a simple point of $\partial X'$}, if there is an open $U\subset \R^{n+1}$ with $x\in U$ such that:
\begin{itemize}
\item [a)] the connected components of $(X'\cap U)\setminus \partial X'$, say $X_1,..., X_r$, are topological manifolds with $\dim X_i=\dim X$, $i=1,...,r$;
\item [b)] $(X_i\cup \partial X')\cap U$ are topological manifolds with boundary. 
\end{itemize}
Let $Smp(\partial X')$ be the set of simple points of $\partial X'$.
\end{definition}
\begin{definition}
Let $X\subset \R^n$ be a subanalytic set such that $0\in X$.
We define $k_X:Smp(\partial X')\to \N$, with $k_X(x)$ is the number of connected components of the germ $(\beta_n^{-1}(X\setminus\{0\}),x)$.
\end{definition}
\begin{remark}
It is clear that the function $k_X$ is locally constant. In fact, $k_X$ is constant in each connected component $C_j$ of $Smp(\partial X')$. Then, we define $k_X(C_j):=k_X(x)$ with $x\in C_j$.
\end{remark}
\begin{remark}\label{density_top}
By Theorems 2.1 and 2.2 in \cite{Pawlucki:1985}, we obtain that ${\rm Smp}(\partial X')$ is an open dense subset of the $(d-1)$-dimensional part of $\partial X'$ whenever $\partial X'$ is a $(d-1)$-dimensional subset, where $d=\dim X$.
\end{remark}
\begin{remark}
The numbers $k_X(C_j)$ are equal to the numbers $n_j$ defined by Kurdyka and Raby \cite{Kurdyka:1989}, pp. 762.
\end{remark}

\begin{definition}
Let $X\subset \R^{n}$ be a real analytic set. We denote by $C_X'$ the closure of the union of all connected components $C_j$ of $Smp(\partial X')$ such that $k_X(C_j)$ is an odd number. We call $C_X'$ the {\bf odd part of $C_X\subset \mathbb{S}^{n}$}.
\end{definition}

\begin{definition}
Let $X\subset \R^{n}$ be a $d$-dimensional real analytic set with $0\in X$ and $\pi:\C^{n}\to \C^{d}$ be a projection such that $\pi^{-1}(0)\cap C(X_{\C},0)=\{0\}$. Let $\pi':\mathbb{S}^n\setminus L\to \mathbb{S}^{d-1}$ given by $\pi'(u)=\frac{\pi(u)}{\|\pi(u)\|}$, where $L=\pi^{-1}(0)$. We define
$$\varphi_{\pi,C_X'}(x):=\#(\pi'^{-1}(x)\cap C_X').$$
In this case, if $\varphi_{\pi,C_X'}(x)$ is constant ${\rm mod\,} 2$ for a generic $x\in\mathbb{S}^{d-1}$, we write $m_{\pi}(C_X'):=\varphi_{\pi,C_X'}(x) {\rm mod\,} 2$, for a generic $x\in\mathbb{S}^{d-1}$.
\end{definition}

\begin{proposition}\label{multcone}
Let $X\subset \R^{n}$ be a $d-$dimensional real analytic set and $0\in X$. Then, $\varphi_{\pi,C_X'}(y)$ is constant for a generic projection $\pi$ and a generic point $y\in\mathbb{S}^{d-1}$. Moreover, $m_{\pi}(C_X')\equiv m(X){\rm mod\,}\,2$.
\end{proposition}
\begin{proof}
Firstly, let us assume that $\dim C_X=d-1$. By Remark \ref{density_top}, $Smp(\partial Y')$ is an open dense subset of the $(d-1)$-dimensional part of $\partial Y'=C_Y\times \{0\}\cong C_Y$.
Let $y\in \mathbb{S}^{d-1}$ be a generic point, $u=\#(\pi^{-1}(ty)\cap X)\equiv m(X) {\rm mod}\,2$, for small enough $t>0$ and $\pi'^{-1}(y)\cap C_X=\pi'^{-1}(y)\cap Smp(\partial X')=\{y_1,...,y_p\}$. Then, we have the following
\begin{equation}
u\equiv \sum\limits _{j=1}^p k_X(y_j).
\end{equation}

In fact, let $\eta,\varepsilon>0$ be small enough numbers such that $C_{\eta,\varepsilon }(y)\cap \pi(br(\pi|_X))=\emptyset $, where $C_{\eta,\varepsilon }(y)=\{v\in \R^d;\, \|v-ty\|\leq \eta t, \,t\in(0,\varepsilon]\}$. Thus, denote the connected components of $(\pi|_X)^{-1}(C_{\eta,\varepsilon }(y))$ by $Y_1,...,Y_u$. Hence, $\pi|_{Y_i}:Y_i\to C_{\eta,\varepsilon }(y)$ is a homeomorphism, for $i=1,...,u$. Thus, for each $i=1,...,u$, there is a unique $\gamma_i\colon (0,\varepsilon)\to Y_i$ such that $\pi(\gamma_i(t))=ty$ for all $t\in (0,\varepsilon)$. We define for each $i=1,...,u$, $\widetilde\gamma_i\colon [0,\varepsilon)\to \overline{\rho^{-1}(Y_i)}$ given by $\widetilde \gamma_i(s)=\lim\limits_{t\to s^+}\rho^{-1}\circ \gamma_i(t)$, for all $s\in [0,\varepsilon)$, where $\rho=\beta_{n}$.

We remark that $\widetilde \gamma_i(0)=\lim\limits _{t\to 0^+}\widetilde \gamma_i(t)\in \{y_1,...,y_p\}$, for all $i=1,...,u$ and, thus, $u\leq \sum\limits _{j=1}^p k_X(y_j)$. Shrinking $\eta$, if necessary, we can suppose that each $C_{Y_i}$ contains at most one $y_j$. Thus, fixed $y_j$ and if $\gamma:[0,\delta)\to X$ is a subanalytic curve such that $\lim\limits _{t\to 0^+}\rho^{-1}\circ \gamma(t)=y_j$, then there exists $\delta_0>0$ such that $\pi(\gamma(t))\in C_{k,\varepsilon }(y)$, for all $t<\delta$. So, there is $i\in \{1,...,u\}$ such that $\gamma(t)\in Y_i$, with $0<t<\delta$. Then, $\widetilde \gamma_i(0)=y_j$ and we obtain the equality $u=\sum\limits _{i=1}^p k_X(y_j)$. Therefore, we obtain
$$
u=\sum\limits _{i=1}^r k_X(C_i)\cdot \#(\pi'^{-1}(y)\cap C_i),
$$
where $C_1,...,C_r$ are the connected components of $Smp(\partial X')$.
Hence, 
$$
\sum\limits _{i=1}^r k_X(C_i)\cdot \#(\pi'^{-1}(y)\cap C_i)\equiv \#(\pi'^{-1}(y)\cap C_X')\, {\rm mod\,} 2
$$
and since $u\equiv m(X)\, {\rm mod\,} 2$, we obtain
$$
m(X)\equiv \#(\pi'^{-1}(y)\cap C_X')\, {\rm mod\,} 2,
$$
for a generic $y\in\mathbb{S}^{d-1}$.

When $\dim C_X<d-1$, we have that $C_X'=\emptyset$ and $\dim C(\pi(X),0)<d$, which implies that there exist $w\in \mathbb{S}^{d-1}$ and small enough numbers $\eta,\varepsilon\in (0,1)$ such that $C_{\eta,\varepsilon }(y)\cap \pi(X)=\emptyset $, where $C_{\eta,\varepsilon }(w)=\{v\in \R^d;\, \|v-tw\|\leq \eta t, \,t\in(0,\varepsilon]\}$. Therefore $\varphi_{\pi,C_X'}(y)=0$ for any point $y\in\mathbb{S}^{d-1}$ and $m(X)\equiv 0\, {\rm mod\,} 2$, since $C_X'=\emptyset$ and $\pi^{-1}(v)\cap Y=\emptyset$, for all $v\in C_{\eta,\varepsilon }(w)$. In particular, $m_{\pi}(C_X')$ is defined and satisfies $0\equiv m_{\pi}(C_X')\equiv m(X)\, {\rm mod\,} 2$.
\end{proof}

\begin{corollary}
Let $X\subset \R^{n}$ be a $d$-dimensional real analytic set and $0\in X$. If $m(X)\equiv 1\, {\rm mod\,} 2$ then $\dim C(X,0)=d$.
\end{corollary}

\subsection{Euler cycles and allowed paths}

\begin{definition}
An $(n-1)$-dimensional subanalytic set $C$ is said to be an {\bf Euler cycle} if it is a closed set and if, for a stratification of $C$ (and hence for any that refines it), the number of $(n-1)$-dimensional strata containing a give $(n-2)$-dimensional stratum in their closure is even.
\end{definition}

\begin{definition}
Let $k\in\mathbb{N}\cup \{\infty,\omega\}$ and let $C\subset \mathbb{S}^n$ be a subset. We say that a point $p\in C$ is a {\bf $C^k$ regular point of $C$} if there exists a open $U\subset \mathbb{S}^n$ that contains $p$ and $C\cap U$ is a $C^k$ submanifold of $\mathbb{S}^n$. We denote by ${\rm Reg}_{k}(C)$ to be the set of all $C^k$ regular points of $C$. We define also ${\rm Sing}_{k}(C)=C\setminus {\rm Reg}_{k}(C)$.
\end{definition}

\begin{definition}
Let $C\subset \mathbb{S}^n$ be an Euler cycle.
A subanalytic $C^1$ path $\gamma \colon [0, 1] \to \mathbb{S}^n$ is said to be an {\bf allowed path for} $C$ if for every $t\in I_{\gamma} = \{t \in [0, 1];\, \gamma(t) \in C\},$ the point $\gamma(t)$ is a $C^1$ regular point of $C$ at which the mapping $\gamma$ is transverse to $C$. In this case, we define
$
lg_C(\gamma)=\#I_{\gamma}.
$
\end{definition}

For $\lambda, \mu\in \mathbb{S}^n\setminus C$, we define
$$
d_C(\lambda;\mu)=\min\{lg_C(\gamma);\, \gamma \mbox{ allowed path joining }\lambda\mbox{ and }\mu\}.
$$

We define the diameter of an Euler cycle $C\subset \mathbb{S}^2$ as the integer
$$
\delta_C=\sup\{d_C(\lambda;\mu);\, \lambda,\mu\in \mathbb{S}^2\setminus C\}.
$$

We have the following result proved by Valette in \cite{Valette:2010}.
\begin{lemma}[Proposition 2.4 and Theorem 4.1 in \cite{Valette:2010}] \label{diameter_mult}
Let $X\subset \R^3$ be a real analytic surface. Then $C_X'$ is an Euler cycle and $\delta_{C_X'}\equiv m(X)\,{\rm mod\,}\,2$.
\end{lemma}

\begin{lemma}[\cite{Valette:2010}, Propositions 2.4, 3.2 and 3.3]\label{path_independence}
Let $X\subset \R^{n}$ be a real analytic hypersurface with $0\in X$. Then, for a generic projection $\pi\colon\R^{n}\to \R^{n-1}$  with $\pi^{-1}(0)\cap \mathbb{S}^n=\{-\lambda, \lambda\}$, $\varphi_{\pi,C_X'}(x)\,{\rm mod\,} 2$ is constant for a generic $x\in\mathbb{S}^{d-1}$ and 
$$m(X)\equiv m_{\pi}(C_X')\equiv d_{C_X'}(\lambda;-\lambda) \,{\rm mod\,} 2.$$ 
\end{lemma}

\subsection{Allowed cycles}
\begin{definition}
Let $C\subset \mathbb{S}^2$ be an $a$-invariant set. We say that an embedding $e\colon\mathbb{S}^1\to C$ is an {\bf allowed embedding}, if $a(Im(e))=Im(e)$ or if there is another embedding $e'\colon\mathbb{S}^1\to C$ such that $e'=a\circ e$ and $Im(e)\cap Im(e')$ is a finite set or empty. A subset $A\subset C$ is called {\bf an allowed set} if there exists a subset $E(A)\subset \{e\colon \mathbb{S}^1\to C; e$ is an embedding$\}$ satisfying the following:
\begin{enumerate}
 \item $A=\bigcup\limits_{e\in E(A)} e(\mathbb{S}^1)$;
 \item if $e_i,e_j\in E(A)$ with $e_i\not =e_j$ then $Im(e_i)\cap Im(e_j)$ is a finite set or empty;
 \item if $e\in E(A)$ with $a(Im(e))\not =Im(e)$ then $a\circ e\in E(A)$.
\end{enumerate}
In this case, we define the {\bf number allowed of circles of $C$} to be $nac(C):=\max \{\# E(A);\, A$ is an allowed set for $C\}$. If $A$ is an allowed set such that $nac(C)=\# E(A)$, we say that $A$ is {\bf a maximal allowed set}. 
\end{definition}
When $X\subset \R^3$ is a real analytic surface, we define {\bf number of allowed circles of $X$} to be $nac(X):=nac(C_X')$.

\begin{remark}\label{finite_max}
Let $C\subset \mathbb{S}^2$ be an $a$-invariant closed subanalytic set. Then $nac(C)<+\infty$.
\end{remark}

\begin{lemma}\label{euler_cycle}
Let $S_1,\cdots,S_r\subset \mathbb{S}^2$ be subanalytic subsets such that each one of them is homeomorphic to $\mathbb{S}^1$. If $S_i\cap S_j$ is a finite set whenever $i\not =j$, then $C=\bigcup\limits_{i=1}^r S_i$ is an Euler cycle.
\end{lemma}
\begin{proof}
It is clear that $C$ is a 1-dimensional closed subanalytic subset. Consider a stratification $\mathcal{S}$ of $C$ such that each $S_i$ has at least two $0$-dimensional strata. Since $S_i$ is homeomorphic to $\mathbb{S}^1$, then it is easy to verify that the number of $1$-dimensional strata containing a give $0$-dimensional stratum in their closure is even.
\end{proof}

\begin{lemma}\label{mult_allowed}
Let $S_1,S_2\subset \mathbb{S}^2$ be subanalytic subsets. Suppose that $S_1$ and $S_2$ are homeomorphic to $\mathbb{S}^1$ and $C=S_1\cup S_2$ is an $a$-invariant set. Then for we have the following:
\begin{itemize}
\item [(a)] If $S_1=S_2$, then $m(C)\equiv 1 \,{\rm mod\,} 2$;
\item [(b)] If $S_1\cap S_2$ is a finite set, then $m(C)\equiv 0 \,{\rm mod\,} 2$.
\end{itemize}
\end{lemma}
\begin{proof}
Let $\lambda\in\mathbb{S}^2\setminus C$ be a generic point. Thus, we consider the stereographic projection $p_{\lambda}:\mathbb{S}^2\setminus \{\lambda \}\to \R^2$ and, then, for a generic point $t\in \R^2$, $\pi^{-1}(t)$ is a ray starting at the origin and $\# \pi^{-1}(t)\cap p_{\lambda}(C)\equiv m(C) \,{\rm mod\,} 2$, where $\pi=\pi'_{\lambda}\circ p_{\lambda}^{-1}:\R^2\setminus \{0\}\to \R^2\setminus \{0\}$. For each $i=1,2$ let $B_i$ be the bounded connected component of $\R^2 \setminus p_{\lambda}(S_i)$. 

Let us choose $\lambda$ such that $0$ belongs to bounded component of $B_1$.
By the proof of Jordan Curve Theorem (see \cite{Tverberg:1980}), $\# \pi^{-1}(t)\cap p_{\lambda}(S_1)\equiv 1 \,{\rm mod\,} 2$. Therefore, we have proven the item (a), since $C=S_1$. 

For item (b), we have two cases:

\noindent (1) $a(S_1)=S_1$. In this case, $a(S_2)=S_2$ and by item (a), $m(S_1)\equiv 1 \,{\rm mod\,} 2$ and $m(S_2)\equiv 1 \,{\rm mod\,} 2$. Therefore, $m(C)\equiv 0 \,{\rm mod\,} 2$, since in this case $m(C)\equiv m(S_1) + m(S_2) \,{\rm mod\,} 2$;;

\noindent (2) $a(S_1)\not =S_1$. In this case, $a(S_1)=S_2$ and we have also that $0$ belongs to bounded component of $B_2$. By the proof of Jordan Curve Theorem once again, $\# \pi^{-1}(t)\cap p_{\lambda}(S_1)\equiv 1$ and $\# \pi^{-1}(t)\cap p_{\lambda}(S_2)\equiv 1$. But $\# \pi^{-1}(t)\cap p_{\lambda}(C)= \# \pi^{-1}(t)\cap p_{\lambda}(S_1)+ \# \pi^{-1}(t)\cap p_{\lambda}(S_2)$ and this finishes the proof.
\end{proof}

\begin{lemma}[\cite{Valette:2010}, Proposition 3.3]\label{mult_dist}
Let $C\subset \mathbb{S}^n$ be an Euler cycle, $\pi\colon\R^{n+1}\to \R^n$ be a generic projection and $\pi^{-1}(0)\cap \mathbb{S}^n=\{-\lambda, \lambda\}$. Then $m_{\pi}(C)\equiv \tilde d_C(-\lambda, \lambda)\,{\rm mod\,} 2$. 
\end{lemma}

% \section{Classification of real analytic curves}
\section{Blow-spherical equivalence}\label{bs_equiv}
\begin{definition}
Let $X\subset\R^n$ and $Y\subset\R^p$ be two subsets containing respectively the origin of $\R^n$ and $\R^p$. 
\begin{itemize}[leftmargin=*]
\item A continuous mapping $\varphi\colon (X,0)\to (Y,0)$, with and $0\not\in \varphi(X\setminus \{0\})$, is a {\bf blow-spherical morphism} (shortened as {\bf blow-morphism}), if the mapping
$$
\beta_p^{-1}\circ \varphi\circ\beta_n\colon X'\setminus \partial X'\to Y'\setminus \partial Y'
$$
extends as a continuous mapping $\varphi'\colon X'\to Y'$.
\item A {\bf blow-spherical homeomorphism} (shortened as {\bf blow-isomorphism}) is a blow-morphism $\varphi\colon (X,0)\to (Y,0)$ such that the extension $\varphi'$ is a homeomorphism. In this case we say that the germs $(X,0)$ and $(Y,0)$ are said {\bf blow-spherical equivalent} or {\bf blow-spherical homeomorphic} (or {\bf blow-isomorphic}). % In this case, we also say that $\varphi\in {\rm BS}_0$.
\end{itemize}
When $\varphi\colon (X,0)\to (Y,0)$ is a  blow-spherical homeomorphism, we denote by $\nu_{\varphi}\colon C_X\to C_Y$ the homeomorphism such that $\varphi'(x,0)=(\nu_{\varphi}(x),0)$ for all $(x,0)\in \partial X'$.
\end{definition}

\begin{remark}
{\rm We have the following.
\begin{enumerate}
\item $id\colon X\to X$ is a blow-spherical homeomorphism for any $X\subset\R^n$ with $0\in X$;
\item Let $X\subset\R^n$, $Y\subset\R^p$ and $Z\subset\R^k$ be subsets containing respectively the origin of $\R^n$, $\R^p$ and $\R^k$. If $f\colon (X,0)\to (Y,0)$ and $g\colon (Y,0)\to (Z,0)$ are blow-spherical morphisms then $g\circ f\colon (X,0)\to (Z,0)$ is a blow-spherical morphism.
\end{enumerate}}
\end{remark}

Thus, we have a category, denoted here by ${\rm BS}_0$, such that their objects are all subsets of Euclidean spaces that contain the origin and their morphisms are all blow-spherical morphisms.

We have also the following result.
\begin{theorem}\label{multiplicities}
Let $\varphi\colon (X,0)\to (Y,0)$ be a blow-spherical homeomorphism. Then, $\varphi'(Smp(\partial X'))=Smp(\partial Y')$ and $k_X(v)=k_{Y}(\varphi'(v))$ for all $v\in Smp(\partial X')$. In particular, $\varphi'(C_X')=C_Y'$.
\end{theorem}
\begin{proof}
% The proof is a direct consequence of the definitions of blow-spherical homeomorphism and $Smp(\partial X')$ as we can see in the next.
Let $v\in Smp(\partial X')$ be a point and let $U \subset X'$ be a small neighborhood of $v$. Since $\varphi'\colon X'\to Y'$ is a homeomorphism, we have that $V=\varphi'(U)$ is a small neighborhood of $\varphi'(v)\in \partial Y'$. Moreover, $\varphi'(U\setminus \partial X')=V\setminus \partial Y'$, since $\varphi'|_{\partial X'}\colon\partial X'\to \partial Y'$ is a homeomorphism, as well. Using once more that $\varphi'$ is a homeomorphism, we obtain that the number of connected components of $U\setminus \partial X'$ is equal to the number of connected components of $V\setminus \partial Y'$, showing that $k_X(v)=k_{Y}(\varphi'(v))$ for all $v\in Smp(\partial X')$. 
In particular, we obtain that $C_Y'=\varphi'(v)(C_X')$. 
\end{proof}

In Sections \ref{section:regularity} and \ref{section:multiplicity}, we give several applications of Theorem \ref{multiplicities}.

% \begin{definition}
% Let $f,g\colon (\R^n,0)\to (\R,0)$ be two functions. We say that $f$ and $g$ are blow-spherical $\mathcal{R}$-equivalent if there is a blow-spherical homeomorphism $\varphi\colon (\R^n,0)\to (\R^n,0)$ such that $g=f\circ \varphi$.
% \end{definition}

\subsection{Blow-spherical equivalence as a new equivalence}\label{new_equiv}
In this Subsection, we show that Blow-Spherical equivalence is different from some other equivalences studied in Singularity Theory.

Let ${\rm Lip}_{0,outer}$ (resp. ${\rm Lip}_{0,inner}$) be the subcategory of ${\rm Lip}$, which their objects are all subsets of Euclidean spaces that contain the origin endowed with the induced metric (resp. intrinsic metric) and their morphisms are all Lipschitz mappings with respect the induced metric (resp. intrinsic metric) such that its inverse image of the origin of the target is only the origin of the source.

Let ${\rm Top}_0$ be the subcategory of ${\rm Top}$ which their objects are all subsets of Euclidean spaces that contain the origin and their morphisms are all continuous mappings such that its inverse image of the origin of the target is only the origin of the source.

\begin{example}
It is clear that ${\rm BS}_0$ is a subcategory of ${\rm Top}_0$.
$X=\{(x,y)\in\R^2;\, y=0\}$ and $Y=\{(x,y)\in\R^2;\, y^2=x^3\}$ are homeomorphic and by Theorem \ref{regular}, $X$ and $Y$ are not blow-spherical homeomorphic, since $X$ is not a $C^1$ submanifold of $\R^2$. Thus, ${\rm BS}_0\not= {\rm Top}_0$.
\end{example}

The above example also shows that blow-spherical equivalence is also different from blow-analytic equivalence (see the definition in \cite{Kuo:1985}), since $\{(x,y)\in\R^2;\, y=0\}$ and $\{(x,y)\in\R^2;\, y^2=x^3\}$ are blow-analytic equivalents (see \cite{KobayashiK:1998}).

% \begin{proposition}
% Let $X, Y\subset \R^n$ be two subanalytic subsets and let $\varphi\colon (X,0)\to (Y,0)$ be a homeomorphism.
% \begin{itemize}
% \item [(a)] If $\varphi$ is bi-Lipschitz and subanalytic, then $\varphi$ is a blow-spherical homeomorphism. 
% \item [(b)] 
% \end{itemize}

% \end{proposition}
% 
% \begin{example}
% $X=\{(x,y)\in\C^2;\, y^2=x^3\}$ and $Y=\{(x,y)\in\C^2;\, y^2=x^5\}$ are blow-spherical homeomorphic (see Theorem 6.1 in \cite{Sampaio:2020}). However, since $X$ and $Y$ have different Puiseux pairs, they are not bi-Lipschitz homeomorphic (see \cite{F,N-P,P-T}). In particular, ${\rm BS}_0\not= {\rm Lip}_{0,outer}$.
% \end{example}

\begin{definition}
For each rational number $\beta\in [1,+\infty )$, we define $X_{\beta}=\{(x,y,z)\in\R^3;\, x^2+y^2=z^{2\beta}\}$ and $X_{\beta}^{\pm}=\{(x,y,z)\in X_{\beta}; \pm z\geq 0\}$.
\end{definition}
We remark that for each $\beta\in [1,+\infty )$, $X_{\beta}^{+}$ is (outer) bi-Lipschitz homeomorphic to $X_{\beta}^{-}$. This implies that if $X_{\beta_1}$ and $X_{\beta_2}$ are inner bi-Lipschitz homeomorphic then $X_{\beta_1}^{+}$ and $X_{\beta_2}^{+}$ are inner bi-Lipschitz homeomorphic as well. However, it is known that $X_{\beta_1}^+$ and $X_{\beta_2}^+$ are not inner bi-Lipschitz homeomorphic whenever $\beta_1\not=\beta_2$ (see \cite{Birbrair:2008}).

\begin{example}\label{blow-iso_no_bi-Lip}
Let $\beta_1,\beta_2\in (1,+\infty )$ be two different rational numbers. Then, the mapping $\varphi\colon X_{\beta_1}\to X_{\beta_2}$ given by $\varphi(x,y,z)=(x,y,{\rm sign}(z) |z|^{\frac{\beta_1}{\beta_2}})$ is a blow-spherical homeomorphism, but $X_{\beta_1}$ and $X_{\beta_2}$ are not inner bi-Lipschitz homeomorphic. In particular, ${\rm BS}_0\not= {\rm Lip}_{0,inner}$ and ${\rm BS}_0\not= {\rm Lip}_{0,outer}$.
\end{example}
In fact, we can find an example with an embedded blow-spherical homeomorphism.
\begin{example}\label{blow-iso_no_bi-Lip_two}
By taking $\beta_1=\frac{5}{2}$ and $\beta_2=\frac{7}{2}$ in Example \ref{blow-iso_no_bi-Lip}, the mapping $\varphi\colon (\R^3,0)\to (\R^3,0)$ given by
$\varphi(x,y,z)=(xg(x,y,z),yg(x,y,z),z)$
is a blow-spherical homeomorphism such that $\varphi(X_{\beta_1})=X_{\beta_2}$,
where 
$$
g(x,y,z)=\left\{\begin{array}{ll}
                        1,&\mbox{ if } z^3\leq x^2+y^2\\
                        \frac{x^2+y^2 +z^4}{z^3(z+1)},&\mbox{ if } z^5<x^2+y^2<z^3\\
                        z,&\mbox{ if }  x^2+y^2\leq z^5.
                      \end{array}\right.
$$
\end{example}

In \cite{KoikeP}, the authors presented the following definition.
\begin {definition}\label{semiline}
We say that a homeomorphism $h : (\R^n,0) \to (\R^n,0)$ satisfies
condition {\em semiline}-(SSP), if $h(\ell)$ has a unique direction
%satisfies condition (SSP) 
for all semilines $\ell$.
\end{definition}
Here, we say that such an $h$ is a {\bf semiline homeomorphism}. 

\begin{example}
The definition of blow-spherical homeomorphism is intrinsic, but a semiline homeomorphism have to be defined in some open neighborhood of $0$. By Theorem 6.1 in \cite{Sampaio:2020}, $X=\{(x,y)\in \C^2;y^2=x^3\}$ and $Y=\{(x,y)\in \C^2;y^2=x^5\}$ are blow-spherical homeomorphic, however, there is no semiline homeomorphism $h\colon (\C^2,0)\to (\C^2,0)$ such that $h(X)=Y$. 
\end{example}

\section{Regularity and classification of real analytic curves}\label{section:regularity}
\subsection{Blow-spherical regularity of real analytic sets}

\begin{definition}\label{def:bs_regular}
A subset $X\subset\R^n$ is called {\bf blow-spherical regular} at $0\in X$ if there is an open neighborhood $U\subset\R^n$ of $0$ such that $X\cap U$ is blow-spherical homeomorphic to an Euclidean ball. 
\end{definition}

\begin{definition}\label{def:lip_regular}
A subset $X\subset\R^n$ is called {\bf Lipschitz regular} (resp. {\bf $C^k$ regular}) at $x_0\in X$ if there is an open neighborhood $U\subset\R^n$ of $x_0$ such that $X\cap U$ is bi-Lipschitz homeomorphic to an Euclidean ball (resp. $X\cap U$ is a $C^k$ submanifold of $\R^n$), where $k\in \mathbb N \cup \{\infty,\omega\}$. 
\end{definition}

\begin{definition}
Let $X\subset \mathbb{R}^n$ and $Y\subset \mathbb{R}^m$ be closed subsets and let $k\in \mathbb{N}\cup \{\infty,\omega\}$. We say that a mapping $f:X\to Y$ is $C^k$ (resp. differentiable at $x\in X$), if there exist an open $U\subset \mathbb{R}^n$ and a mapping $F:U\to \mathbb{R}^m$ such that $x\in U$, $F|_{X\cap U}=f|_{X\cap U}$ and $F$ is $C^k$ (resp. differentiable at $x$). 
\end{definition}

\begin{proposition}\label{analytic_equiv}
Let $X$ and $Y$ be two real analytic sets and let $\varphi\colon (X,0)\to (Y,0)$ be a real analytic diffeomorphism. Then $m(X)=m(Y)$.
\end{proposition}
\begin{proof}
 We have that the complexification of $\varphi$, denoted by $\varphi_{\C}$, is a complex diffeomorphism between $X_{\C}$ and $Y_{\C}$. Thus, by Proposition in (\cite{Chirka:1989}, Section 11, p. 120), $m(X_{\C},0)=m(Y_{\C},0)$. Therefore, $m(X)=m(Y)$.
\end{proof}

In fact, we can obtain a stronger version of the above result.
\begin{proposition}\label{smooth_equiv}
Let $X$ and $Y$ be two real analytic sets and let $\varphi\colon (X,0)\to (Y,0)$ be a $C^{\infty }$ diffeomorphism. Then $m(X)=m(Y)$.
\end{proposition}
\begin{proof}
By Proposition 1.1 in \cite{Ephraim:1973}, $(X,0)$ and $(Y,0)$ are real analytic diffeomorphic and by Proposition \ref{analytic_equiv}, $m(X)=m(Y)$.
\end{proof}
Thus, we obtain the following result.
\begin{proposition}\label{mult_um}
Let $X\subset \R^n$ be a real analytic set. Then, the below statements are equivalent. 
\begin{itemize}
\item [(1)] $m(X)=1$;
\item [(2)] $X$ is $C^{\omega }$ regular at $0$;
\item [(3)] $X$ is $C^{\infty }$ regular at $0$.
\end{itemize}
\end{proposition}
\begin{proof}
It is clear that $(2)\Rightarrow (3)$ and by Propositions \ref{analytic_equiv} and \ref{smooth_equiv}, we have that $(2)\Rightarrow (1)$ and $(3)\Rightarrow (1)$.
Since ${\rm Sing}(X)_{\C}={\rm Sing}(X_{\C})$ (see \cite[Lemma 9]{Whitney:1957}) and $m(X_{\C},0)=1$ if and only if $X_{\C}$ is the germ of a complex analytic submanifold, we obtain that if $m(X):=m(X_{\C},0)=1$ then $({\rm Sing}(X_{\C}),0)=\emptyset $ and, thus, $({\rm Sing}(X),0)=\emptyset $, which implies that $X$ is $C^{\omega }$ regular at $0$. Therefore, $(1)\Rightarrow (2)$, which finishes the proof.
\end{proof}

However, Proposition \ref{mult_um} does not hold true when we consider $C^{1}$ instead $C^{\infty }$.
\begin{example}
Let $X=\{(x,y)\in\R^2;\, y^3=x^4\}$. Then, $X$ is $C^1$ diffeomorphic to $\{(x,y)\in\R^2;\, y=0\}$. Moreover, $X$ is $C^1$ regular at $0$, but it is not $C^{\infty }$ regular at $0$. 
\end{example}
The above example tells us also that Propositions \ref{analytic_equiv} and \ref{smooth_equiv} do not hold true when we consider $C^{1}$ instead $C^{\omega}$ or $C^{\infty }$, since $X=\{(x,y)\in\R^2;\, y^3=x^4\}$ is $C^1$ diffeomorphic to $Y=\{(x,y)\in\R^2;\, y=0\}$, but $m(X)=3$ and $m(Y)=1$.

% \begin{lemma}[\cite{FernandesS:2017}]\label{half-line}
% Let $X\in \R^{n+1}$ an analytic curve and $x_0\in X$ a point non-isolated. If $X$ is a topological manifold, then $C(X,x_0)$ is a half-line nor is $(X,x_0)$ is $C^1$. In particular, $(X,x_0)$ is $C^1$ if, and only if, $C(X,x_0)$ is a line.
% \end{lemma}

\begin{definition}
Let $X\subset \mathbb{R}^n$ and $Y\subset \mathbb{R}^m$ be closed subsets such that $0\in X\times Y$. We say that $(X,0)$ and $(Y,0)$ are {\bf differentiable equivalent} if there exists a homeomorphism $\varphi\colon (X,0)\to (Y,0)$ such that $\varphi$ and $\varphi^{-1}$ are differentiable at $0$. In this case, we say that $\varphi$ is a {\bf differentiable equivalence} (between $(X,0)$ and $(Y,0)$).
\end{definition}

\begin{lemma}[Proposition 1 in \cite{Sampaio:2019}]\label{diffeo_blow}
Let $X, Y\subset \R^m$ be subsets. If $(X,0)$ and $(Y,0)$ are differentiable equivalent at the origin, then $(X,0)$ and $(Y,0)$ are blow-spherical homeomorphic.
\end{lemma}

\begin{lemma}[Proposition 3.3 in \cite{Sampaio:2020}]\label{defi_diferencial}
If $X$ and $Y$ are blow-spherical homeomorphic, then $C(X,0)$ and $C(Y,0)$ are also blow-spherical homeomorphic.
\end{lemma}

\begin{lemma}[Milnor \cite{Milnor:1968}, Lemma 3.3]\label{branches}
Let $V\subset \R^{n}$ be a real analytic curve and $x_0\in V$ a non-isolated point. Then, there are an open neighborhood $U\subset \R^{n}$ of $x_0$ and $\Gamma_1,...,\Gamma_r\subset\R^{n}$ such that $\Gamma_i\cap\Gamma_j=\{x_0\}$ whenever $i\not=j$ and  
$$V\cap U=\bigcup\limits _{i=1}^r\Gamma_i.$$ 
Moreover, for each $i\in\{1,...,r\}$, there is an analytic homeomorphism $\gamma_i\colon(-\varepsilon ,\varepsilon )\to \Gamma_i$.
\end{lemma}
Each $\Gamma_i$ in the above proposition is called {\bf an analytic branch of $X$ at $x_0$} and $\Gamma_1,...,\Gamma_r$ is called {\bf a decomposition in analytic branches for $(X,x_0)$}.

\begin{theorem}\label{half-line}
Let $\gamma\colon (-\varepsilon,\varepsilon)\to \R^n$ be an analytic curve and $\Gamma=\gamma((-\varepsilon,\varepsilon))$. Suppose that $\gamma\colon (-\varepsilon,\varepsilon)\to \Gamma$ is a homeomorphism and $\gamma(0)=0$. Then the following statements are equivalent:
\begin{enumerate}
 \item $\Gamma$ is blow-spherical regular at $0$;
 \item $C(\Gamma,0)$ is homeomorphic to $\R$;
 \item $C(\Gamma,0)$ is a real line;
 \item ${\rm ord}_0 \gamma$ is an odd number and, in particular, 
 $
 \lim\limits_{t\to 0^+}\frac{\gamma(t)}{\|\gamma(t)\|}=-\lim\limits_{t\to 0^-}\frac{\gamma(t)}{\|\gamma(t)\|};
 $
 \item $\Gamma$ is $C^1$ regular at $0$;
 \item $\Gamma$ is $C^1$ Lipschitz regular at $0$.
\end{enumerate}
\end{theorem}
\begin{proof}
Since $\gamma'$ also is an analytic curve, by shrinking $\varepsilon$, if necessary, we can assume that $\gamma'(t)\not=0$ for all $t\in (-\varepsilon,\varepsilon)\setminus \{0\}$ and changing $\gamma$ by $\gamma\circ r$, where $r\colon(-1,1)\to (-\varepsilon,\varepsilon)$ is given by $r(t)=\varepsilon t$, we can suppose that $\varepsilon =1$. Moreover, for $k={\rm ord}_0 \gamma$, there exist $w\in\R^n$ such that $\gamma(t)=t^kw+o(t^k)$ and an analytic curve $\alpha \colon(-1 ,1 )\to \R^{n}$ satisfying $\alpha(0)\not=0$ and $\gamma(t)=t^k\alpha(t)$, for all $t\in (-1,1)$. Hence, we have the following
$$
u=\lim\limits _{t\to 0^+}\frac{\gamma(t)}{\|\gamma(t)\|}=\frac{w}{\|w\|}
$$
and
$$
v=\lim\limits _{t\to 0^-}\frac{\gamma(t)}{\|\gamma(t)\|}=(-1)^k\frac{w}{\|w\|}.
$$

\noindent $(1)\Rightarrow (2)$. If $\Gamma$ is blow-spherical regular at $0$ then $(\Gamma,0)$ and $(\R,0)$ are blow-spherical homeomorphic. Thus, by Lemma \ref{defi_diferencial}, $C(X,0)$ is homeomorphic to $\R$.

\noindent $(2)\Rightarrow (4)$ and $(3)\Rightarrow (4)$. Suppose that $C(\Gamma,0)$ is homeomorphic to $\R$.
Suppose that $k$ is an even natural number. Then by definitions of $u$ and $v$, we have $u=v$ and in this case $C(\Gamma,0)=\{\lambda u;\, \lambda \geq 0\}$. Therefore, it is clear that $C(\Gamma,0)$ is not homeomorphic to $\R$, which is a contradiction.

\noindent $(4)\Rightarrow (5)$. Suppose that $k$ is an odd number. Then, the curve $\beta\colon (-1,1)\to X$ given by $\beta(s)=\gamma(s^{\frac{1}{k}})$ is well defined. Moreover, $\beta\colon (-1,1)\to X\cap U$ is a homeomorphism, since the function $h\colon (-1,1)\to (-1,1)$ given by $h(s)=s^{\frac{1}{k}}$ and $\gamma$ are homeomorphism. Thus, we obtain that $\beta(s)=s\alpha(s^{\frac{1}{k}})$ and, in this form, we obtain that $\beta$ is a $C^1$ function with $\beta'(0)\not =0$. Therefore, $\Gamma$ is $C^1$ regular at $0$.

\noindent $(5)\Rightarrow (1)$. If $\Gamma$ is $C^1$ regular at $0$, then $(\Gamma,0)$ and $(\R,0)$ are $C^1$ diffeomorphic and by Lemma \ref{diffeo_blow}, $(\Gamma,0)$ and $(\R,0)$ are blow-spherical homeomorphic.

\noindent $(5)\Rightarrow (6)$. If $\Gamma$ is $C^1$ regular at $0$, then $(\Gamma,0)$ and $(\R,0)$ are $C^1$ diffeomorphic. Therefore, $(\Gamma,0)$ and $(\R,0)$ are bi-Lipschitz homeomorphic.

\noindent $(6)\Rightarrow (2)$. If $\Gamma$ is Lipschitz regular at $0$, then $(\Gamma,0)$ and $(\R,0)$ are bi-Lipschitz homeomorphic. By Theorem 3.2 in \cite{Sampaio:2016}, $C(X,0)$ is homeomorphic to $\R$.
\end{proof}
\begin{corollary}\label{regular}
Let $X\subset \R^n$ be a real analytic curve. Then, $X$ is blow-spherical regular at $0$ if and only if $X$ is $C^1$ regular at $0$.
\end{corollary}
\begin{proof}
Suppose that $X$ is blow-spherical regular at $0$.
By Lemma \ref{branches}, there are an open neighborhood $U\subset \R^{n}$ of $0$ and $\Gamma_1,...,\Gamma_r\subset\R^{n}$ such that $\Gamma_i\cap\Gamma_j=\{0\}$, if $i\not=j$ and  
$$X\cap U=\bigcup\limits _{i=1}^r\Gamma_i.$$ 
Moreover, for each $i\in\{1,...,r\}$, there is an analytic homeomorphism $\gamma_i\colon(-\varepsilon ,\varepsilon )\to \Gamma_i$.
Since $X$ is blow-spherical regular at $0$, then $r=1$, $X\cap U= \Gamma_1$ and $\gamma_1\colon(-\varepsilon ,\varepsilon )\to X\cap U$ is an analytic homeomorphism. Since $X$ is blow-spherical regular at $0$, then by Theorem \ref{half-line}, $X$ is $C^1$ regular at $0$.

Reciprocally, if $X$ is $C^1$ regular at $0$, by Lemma \ref{diffeo_blow}, we obtain that $X$ is blow-spherical regular at $0$.
\end{proof}
This result is sharp. Firstly, the hypothesis of $X$ to be blow-spherical regular at $0$ cannot be removed.
\begin{example}
Let $X=\{(x,y)\in\R^2;\, y^2=x^3\}$. Then $X$ is homeomorphic to $\R$, but is not $C^1$. In fact, $X$ is a topological submanifold of $\R^2$.
\end{example}
% In particular, blow-analytical equivalence defined in \cite{KoikeP:2013} does not imply blow-spherical equivalence, since $\{(x,y)\in\R^2;\, y^2=x^3\}$ and $\{(x,y)\in\R^2;\, y=0\}$ are blow-analytic equivalent, but they are not blow-spherical homeomorphic.

Secondly, the hypothesis of $X$ to be a curve (i.e. $\dim X=1$) also cannot be removed.
\begin{example}\label{naosuave_blow}
Let $V=\{(x,y,z)\in\R^3;\, z^3=x^5y+xy^5\}$. Then $C(V,0)=\{z=0\}$ is a plane and $V$ is a topological submanifold of $\R^3$. Moreover, $V$ is graph of a differentiable function at the origin and, thus, $V$ is blow-spherical regular at $0$. However, $V$ is not $C^1$ regular at $0$.
\end{example}
Finally, the hypothesis of $X$ to be an analytic set cannot be removed as well.
\begin{example}
The set $V=\{(x,y)\in\R^2;\, y=|x|\}$ is semi-algebraic and the mapping $\varphi:\R\to V$ given by $\varphi(x)=(x,|x|)$ is a blow-spherical homeomorphism and, in particular, $V$ is blow-spherical regular at $0$, but clearly $V$ is not $C^1$ regular at $0$.
\end{example}

% \subsection{Characterization of curves}

% Let $X,Y\subset \R^n$ be two real analytic curves. Let $X_1,...,X_r\subset \C^n$ be the branches of $X$ and let $Y_1,...,Y_s\subset \C^n$ be the branches of $Y$. Then, we have the following
% \begin{corollary}
% $X$ and $Y$ are blow-spherical equivalent if and only if there is a bijection $\sigma:\{1,...,r\}\to\{1,...,s\}$ such that there is a homeomorphism $h\colon (C(X,0),0)\to (C(Y,0),0)$ satisfying 
% $$h(C(X_i,0))=C(Y_{\sigma(i)},0), \mbox{ for all }i=1,...,r.$$
% \end{corollary}

\subsection{Classification of the real analytic curves module blow-isomorphisms}
Let $p_1\colon \mathbb Z_{>0}\times \mathcal F(\mathbb Z_3;\mathbb Z_{\geq 0})\to \mathbb Z_{>0}$ and $p_2\colon \mathbb Z_{>0}\times \mathcal F(\mathbb Z_3;\mathbb Z_{\geq 0})\to \mathcal F(\mathbb Z_3;\mathbb Z_{\geq 0})$ be the canonical projections, where $\mathcal F(\mathbb Z_3;\mathbb Z_{\geq 0})$ denotes the set of all non-null functions from $\mathbb Z_3\cong\{-1,0,1\}$ to $\mathbb Z_{\geq 0}$.
Let $\mathcal A$ be the subset of $\mathbb Z_{>0}\times \mathcal F(\mathbb Z_3;\mathbb Z_{\geq 0})$ formed by finite and non-empty subsets $A$ satisfying the following:
\begin{itemize}
\item [i)] $p_1(A)=\{1,...,N\}$ for some $N\in \mathbb Z_{>0}$; 
\item [ii)] $p_2(p_1^{-1}(\ell )\cap A)=\{r_{\ell}\}$ and $r_{\ell }(-1)\leq r_{\ell}(1)$ for all $\ell\in \{1,...,N\}$;
\item [iii)] $r_{\ell }(0)\leq r_{\ell+1}(0)$ for all $\ell \in \{1,...,N-1\}$. Moreover, if $r_{\ell }(0)= r_{\ell+1}(0)$ then $\sum\limits_{i=-1}^1r_{\ell }(i)\leq \sum\limits_{i=-1}^1r_{\ell+1}(i)$.
\end{itemize}
For a set $A\in \mathcal A$ as above, we define the following curves: For $j\in \{-1,0, 1\}$ and $r_{\ell }(j)>0$

$$
X_{A,j}=\bigcup\limits_{\ell=1}^N \{(x,y)\in \R^2; \displaystyle\prod \limits_{r=1}^{r_{\ell }(j)}((y-\ell x)^{2}-jr(y+\ell x)^{3})=0\},
$$
when $j\in \{-1, 1\}$ and 
$$
X_{A,j}=\bigcup\limits_{\ell=1}^N \{(x,y)\in \R^2; \displaystyle\prod \limits_{r=1}^{r_{\ell }(0)}((y-\ell x)^{3}-r(y+\ell x)^{4})=0\}
$$
when $j=0$.
Moreover, if $r_{\ell }(j)=0$ we define $X_{A,j}=\{0\}$.
We define the realization of $A$ to be the curve $X_A:=X_{A,-1}\cup X_{A,0}\cup X_{A,1}$.

\begin{remark}\label{homeo_to_cusp}
Let $\gamma\colon (-\varepsilon,\varepsilon)\to \R^n$ be an analytic curve and $\Gamma=\gamma((-\varepsilon,\varepsilon))$. Suppose that $\gamma\colon (-\varepsilon,\varepsilon)\to \Gamma$ is a homeomorphism and $\gamma(0)=0$. Then, by Theorem \ref{half-line} and the definition of blow-spherical equivalence, $\Gamma$ and $\{(x,y)\in \R^2;\, y^2=x^3\}$ are blow-spherical homeomorphic if and only if $C(\Gamma,0)$ is a half-line. 
\end{remark}

% Thus, by Proposition \ref{half-line}, Lemmas \ref{path_independence} and \ref{mult_dist}, \ref{curve} and definition of $\mathcal A$, it is not hard to verify the following classification:
\begin{definition}
Let $X\subset \R^n$ and $\tilde X\subset \R^m$ be two analytic sets. We say that $(X,0)$ is {\bf branch by branch blow-spherical homeomorphic} to $(\tilde X,0)$ if there are decompositions in analytic branches $\Gamma_1,...,\Gamma_r$ for $(X,0)$ and $\tilde\Gamma_1,...,\tilde \Gamma_r$ for $(\tilde X,0)$ and there is a blow-spherical homeomorphism $\varphi\colon (X,0)\to (\tilde X,0)$ such that $\varphi((\Gamma_i,0))=(\tilde\Gamma_{i},0)$ for all $i\in \{1,...,r\}$.
\end{definition}

Therefore we obtain the following classification.
\begin{theorem}\label{curve_classif}
For each real analytic curve $X\subset \R^n$ such that $0\in X$ there exists a unique set $A\in \mathcal A$ such that $(X_A,0)$ is branch by branch blow-spherical homeomorphic to $(X,0)$.
\end{theorem}
\begin{proof}
 Given a real analytic curve $X\subset \R^2$, by Lemma \ref{branches}, there are an open neighborhood $U\subset \R^{n}$ of $0$ and $\Gamma_1,...,\Gamma_r\subset\R^{n}$ such that $\Gamma_i\cap\Gamma_j=\{0\}$ whenever $i\not=j$ and  
$$X\cap U=\bigcup\limits _{i=1}^r\Gamma_i,$$
where for each $i\in\{1,...,r\}$, there is an analytic homeomorphism $\gamma_i\colon(-\varepsilon ,\varepsilon )\to \Gamma_i$. Let $L_1,...,L_N$ be all the lines such that each $L_j$ contains the tangent cone of some $\Gamma_i$. By reordering the indices, if necessary, we may assume that for each $\ell\in \{1,...,N-1\}$ we have
\begin{equation}\label{eq_ordering}
\#\{i;C(\Gamma_i,0)= L_{\ell}\}\leq \#\{i;C(\Gamma_i,0)= L_{\ell+1}\}
\end{equation}
and, moreover, if happens the equality in (\ref{eq_ordering}) then 
$$
\#\{i;C(\Gamma_i,0)\subset L_{\ell}\}\leq \#\{i;C(\Gamma_i,0)\subset L_{\ell+1}\}.
$$
Thus, for each $\ell\in \{1,...,N\}$ we write $L_{\ell}=L_{\ell}^-\cup L_{\ell}^+$, where $L_{\ell}^-$ and $ L_{\ell}^+$ are half-lines such that $L_{\ell}^-\cap L_{\ell}^+=\{0\}$ and
$$
\#\{i;C(\Gamma_i,0)= L_{\ell}^-\}\leq \#\{i;C(\Gamma_i,0)= L_{\ell}^+\}.
$$
and, thus, we define the function $r_{\ell}\colon \{-1,0,1\}\to \mathbb{Z}_{\geq 0}$ given by
$r_{\ell}(-1)=\#\{i;C(\Gamma_i,0)= L_{\ell}^-\}$, $r_{\ell}(0)=\#\{i;C(\Gamma_i,0)= L_{\ell}\}$ and $r_{\ell}(1)=\#\{i;C(\Gamma_i,0)= L_{\ell}^+\}$. 

Then $A=\bigcup\limits_{\ell=1}^N\{(\ell,r_{\ell})\}\in \mathcal{A}$ and it is not hard to verify that $(X_A,0)$ is branch by branch blow-spherical homeomorphic to $(X,0)$. The uniqueness of $A$ follows directly from the definition of $\mathcal{A}$.
% Let $S_1,...,S_s$ be all the lines such that each $S_j$ is the tangent cone of some $\Gamma_i$. We write 
% $$
% \bigcup \limits_{\ell =1}^{s}S=\bigcup \limits_{\ell =1}^{m}S_j.
% $$
% 
% decompose $X$ in branches in the following way
%  $$
%  X=\bigcup \limits_{\ell =1}^{N}\bigcup \limits_{\ell =1}^{N}\bigcup \limits_{\ell =1}^{N}
%  $$
\end{proof}

\begin{definition}
The {\bf real blow-spherical tree at $0$ of $X$} is the rooted tree with a root corresponding to the curve, with vertices $L_k^1$'s corresponding to the tangent half-lines at $0$ $\ell_k$'s and for each vertex $L_j$, we put vertices $V_j^{1},...,V_j^{k_j}$  corresponding to the branches at infinity $X_j^{i}$'s which are tangent at $0$ to $\ell_j$. We put an edge joining the vertices $L_j$ and $V_j^{i}$. We put also edges joining each vertex $L_j$ with the root.
 \end{definition}

% 
% \begin{neuralnetwork}
% \newcommand{\x}[2]{$X_#2$}
% \newcommand{\y}[2]{$L_#2$}
% \newcommand{\m}[4] {$m_{#2}$}
% \setdefaultlinklabel{\m}
% \inputlayer[count=3,bias=false, text=\x]{}
% \outputlayer[count=2, text=\y]{}
% \link[style={}, labelpos=near start, from layer=0, from node=1, to layer=1, to node=1]
% \link[style={}, labelpos=near start, from layer=0, from node=2, to layer=1, to node=1]
% \link[style={}, labelpos=near start, from layer=0, from node=3, to layer=1, to node=2]
% \link[style={}, labelpos=near start, from layer=0, from node=1, to layer=0, to node=2]
% \end{neuralnetwork}

See some examples of real blow-spherical trees at $0$ in Figures \ref{fig:realbstree} and \ref{fig:realbstree_two}.

\begin{figure}[h!]
     \centering
 \begin{tikzpicture}[line cap=round,line join=round,>=triangle 45,x=0.7cm,y=0.7cm]
\clip(-4,-0.5) rectangle (4,5);
\draw [line width=1pt] (0,0)-- (0,2);
\draw [line width=1pt] (0,2)-- (-1,4);
% \draw [->,line width=1pt] (-1,2)-- (-3,3);
\draw [line width=1pt] (0,2)-- (1,4);
\draw [line width=1pt] (0,0) circle (4pt);
\begin{scriptsize}
\draw [fill=black] (0,0) circle (2pt);
\draw [fill=black] (0,2) circle (2pt);
\draw [fill=black] (-1,4) circle (2pt);
\draw [fill=black] (1,4) circle (2pt);
\end{scriptsize}
\end{tikzpicture}
\caption{Real blow-spherical tree at $0$ of the curve $y^2-x^3=0$.}
     \label{fig:realbstree}
\end{figure}
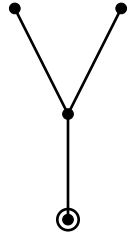

 \begin{figure}[h!]
     \centering
 \begin{tikzpicture}[line cap=round,line join=round,>=triangle 45,x=0.7cm,y=0.7cm]
\clip(-4,-0.5) rectangle (4,5);
\draw [line width=1pt] (0,0)-- (2,2);
\draw [line width=1pt] (0,0)-- (-2,2);
\draw [line width=1pt] (2,2)-- (1,4);
\draw [line width=1pt] (2,2)-- (2,4);
\draw [line width=1pt] (2,2)-- (3,4);
\draw [line width=1pt] (-2,2)-- (-3,4);
% \draw [->,line width=1pt] (-1,2)-- (-3,3);
\draw [line width=1pt] (0,0) circle (4pt);
\begin{scriptsize}
\draw [fill=black] (0,0) circle (2pt);
\draw [fill=black] (2,2) circle (2pt);
\draw [fill=black] (-2,2) circle (2pt);
\draw [fill=black] (1,4) circle (2pt);
\draw [fill=black] (2,4) circle (2pt);
\draw [fill=black] (3,4) circle (2pt);
\draw [fill=black] (-3,4) circle (2pt);
\end{scriptsize}
\end{tikzpicture}
\caption{Real blow-spherical tree at $0$ of the curve $y(y^2-x^3)=0$.}
     \label{fig:realbstree_two}
\end{figure}

It is a direct consequence from Theorem \ref{multiplicities} and the definitions of blow-spherical homeomorphisms and blow-spherical trees at 0 of curves, the following result:
\begin{theorem}\label{curve_classif_two}
Let $X,\widetilde X\subset \R^2$ be two real analytic curves. Let ${\rm Smp}(\partial X')=\{a_1, a_2, ... , a_r\}$ and ${\rm Smp}(\partial \widetilde X')=\{\widetilde a_1, \widetilde a_2, ... , \widetilde a_s\}$. 
Then the following statements are equivalents:
\begin{enumerate}
\item [(1)] $(X,0)$ and $(\widetilde X,0)$ are blow-spherical homeomorphic;
\item [(2)] There is a bijection $\sigma\colon \{1,...,r\}\to \{1,...,s\}$ 
such that for each $i\in \{1,...,r\}$, $k_{X}(a_i)=k_{\widetilde X}(\widetilde a_{\sigma(i)})$;
\item [(3)] There is an isomorphism between the real blow-spherical trees at $0$ of $X$ and $\widetilde{X}$.
\end{enumerate}
\end{theorem}

\section{Invariance of the multiplicity by blow-isomorphisms}\label{section:multiplicity} 
\subsection{Blow-spherical differential invariance of the multiplicity}\label{subsec:blow-spherical-diff}
\subsubsection{Real case}\label{subsubsec:blow-spherical-diff-real}
% In this Subsection, we proof a generalization of the real version of Gau-Lipman's Theorem.

\begin{definition}
Let $X,Y\subset \R^n$ be two sets with $0\in X\cap Y$ and let $\varphi\colon X\to Y$ be a blow-spherical homeomorphism. We say that $\varphi$ is blow-spherical differentiable (at $0$) if there is a linear isomorphism $\phi\colon \R^n\to  \R^n$ such that $\nu_{\varphi}(x)=\frac{\phi(x)}{\|\phi(x)\|}$ for all $x\in C_X$.
\end{definition}

\begin{theorem}\label{blow-spherical-diff}
Let $X,Y\subset \R^n$ be two analytic sets with $0\in X\cap Y$. If there exists a blow-spherical differentiable mapping $\varphi\colon X\to Y$, then $m(X)\equiv m(Y)\,{\rm mod\,} 2$
\end{theorem}
\begin{proof}
Since $\varphi$ is a blow-spherical differentiable mapping, then there exists an $\R$-linear isomorphism $\phi\colon\R^{n} \to \R^{n}$, such that $\nu_{\varphi}(x)=\frac{\phi(x)}{\|\phi(x)\|}$ for all $x\in C_X$. Then $A=\phi(X)$ is a real analytic set and by Proposition \ref{analytic_equiv}, $m(X)=m(A)$.

Thus, it is enough to show that $m(Y)\equiv m(A)\, {\rm mod\,} 2$. In order to do this, we consider the mapping $\psi\colon (Y,0)\to (A,0)$ given by $\psi=\phi\circ \varphi^{-1}$. Thus, $\psi$ is a blow-spherical homeomorphism such that $\nu_{\psi}=id$, i.e., the mapping $\psi': Y'\to A'$ is given by
$$
\psi'(x,t)=\left\{\begin{array}{ll}
\left(\frac{\psi(tx)}{\|\psi(tx)\|},\|\psi(tx)\|\right),& t\not=0\\
(x,0),& t=0.
\end{array}\right.
$$
Then, we obtain that $C_Y'=\psi'(C_Y')=C_{A}'$ and, thus, $m(C_Y')\equiv m(C_{A}')\,{\rm mod\,} 2$. By Proposition \ref{multcone}, we obtain $m(Y) \equiv m(A)\,{\rm mod\,} 2$, which finishes the proof.
\end{proof}

As a consequence, we obtain the main result of \cite{Sampaio:2020c}.
\begin{corollary}[Theorem 3.1 in \cite{Sampaio:2020c}]\label{gen_real-gau-lipman}
Let $X, Y \subset  \R^n$ be two real analytic sets  with $0\in X\cap Y$. Assume that there exists a mapping $\varphi\colon (\R^n,0)\to (\R^n,0)$ such that $\varphi\colon (X,0)\to (Y,0)$ is a homeomorphism. If $\varphi$ has a derivative at the origin and $D\varphi_0\colon \R^n\to \R^n$ is an isomorphism, then $m(X)\equiv m(Y)\, {\rm mod\,} 2$.
\end{corollary}

In particular, we obtain the real version of Gau-Lipman's Theorem in \cite{Gau-Lipman:1983}.
\begin{corollary}[Corollary 3.2 in \cite{Sampaio:2020c}]\label{real-gau-lipman}
Let $X,Y\subset \R^n$ be two analytic sets  with $0\in X\cap Y$. If there exists a homeomorphism $\varphi\colon (\R^n,X,0)\to (\R^n,Y,0)$ such that $\varphi$ and $\varphi^{-1}$ are differentiable at $0$, then $m(X)\equiv m(Y)\,{\rm mod\,} 2$
\end{corollary}
\vspace{0.1cm}

\subsubsection{Complex case}\label{subsubsec:blow-spherical-diff-complex}\, 

When $X$ is a complex analytic set, there is a complex analytic set $\Sigma$ with $\dim \Sigma <\dim X$ such that for each irreducible component $X_j$ of tangent cone $C(X,0)$, $X_j\setminus \Sigma$ intersects only one connected component $C_i$ of $Smp(\partial X')$ (see \cite{Chirka:1989}, pp. 132-133). Then, we define $k_X(X_j):=k_X(C_i)$.

\begin{remark}[{\cite[p. 133, Proposition]{Chirka:1989}}]\label{multip-rel}
{\rm Let $X$ be a complex analytic set of $\C^n$ with $0\in X$ and let $X_1,...,X_r$ be the irreducible components of $C(X,0)$. Then} 
\begin{equation*}\label{mult_rel}
m(X,0)=\sum_{j=1}^rk_X(X_j)\cdot m(X_j,0).
\end{equation*}
\end{remark}
\begin{theorem}\label{blow-spherical-diff-complex}
Let $X,Y\subset \C^n$ be two complex analytic sets with $0\in X\cap Y$. If there exists a blow-spherical differentiable mapping $\varphi\colon X\to Y$, then $m(X,0)= m(Y,0)$.
\end{theorem}
\begin{proof}
Since $\varphi$ is a blow-spherical differentiable mapping, then there exists an $\R$-linear isomorphism $\phi\colon\R^{2n} \to \R^{2n}$ such that $\nu_{\varphi}(x)=\frac{\phi(x)}{\|\phi(x)\|}$ for all $x\in C_X$. Since $\nu_{\varphi}$ is a homeomorphism between $C_X$ and $C_Y$, then $\phi$ is a homeomorphism between $C(X,0)$ and $C(Y,0)$, which implies that $\phi$ maps bijectively the irreducible components of $C(X,0)$ over the irreducible components of $C(Y,0)$ (see Lemma A.8 in \cite{Gau-Lipman:1983} or Proposition 2 in \cite{Sampaio:2019}). Let $X_1,...,X_r$ and $Y_1,...,Y_r$ be the irreducible components of $C(X,0)$ and $C(Y,0)$, respectively, such that $Y_j=\phi(X_j)$, $j=1,...,r$. 
By Proposition 4 in \cite{Sampaio:2019}, $m(X_j,0)=m(Y_j,0)$, for all $j=1,...,r$.
By Theorem \ref{multiplicities}, we obtain $k_X(X_j)=k_Y(Y_j)$, for all $j=1,...,r$. The proof follows from Remark \ref{multip-rel}.
\end{proof}

\begin{corollary}[Theorem 4.1 in \cite{Sampaio:2020c}]\label{generalization_gau-lipman_thm}
Let $X, Y \subset  \C^n$ be two complex analytic sets with $0\in X\cap Y$. Assume that there exists a mapping $\varphi\colon (\C^n,0)\to (\C^n,0)$ such that $\varphi|_X\colon (X,0)\to (Y,0)$ is a homeomorphism. If $\varphi$ has a derivative at the origin (as a mapping from $(\R^{2n},0)$ to $(\R^{2n},0)$) and $D\varphi_0\colon \R^{2n}\to \R^{2n}$ is an isomorphism, then $m(X,0)= m(Y,0)$.
\end{corollary}

\begin{corollary}[Gau-Lipman's Theorem \cite{Gau-Lipman:1983}]\label{complex_gau-lipman_thm}
Let $X ,Y\subset \C^n$ be two complex analytic sets. If there exists a homeomorphism $\varphi\colon (\C^n,X,0)\to (\C^n,Y,0)$ such that $\varphi$ and $\varphi^{-1}$ have a derivative at the origin (as mappings from $(\R^{2n},0)$ to $(\R^{2n},0)$), then $m(X,0)= m(Y,0)$.
\end{corollary}
In the next example, we show that Theorem \ref{blow-spherical-diff-complex} is really a generalization of Gau-Lipman's Theorem. 
\begin{example}
Let $X=\{(x,y)\in \C^2;y^2=x^3\}$ and $Y=\{(x,y)\in \C^2;y^2=x^5\}$.
By Theorem 6.1 and its proof in \cite{Sampaio:2020} there exists a blow-spherical homeomorphism $\varphi\colon X\to Y$ such that $\nu_{\varphi}=id_{C_X}$. Thus, $\varphi$ is a blow-spherical differentiable mapping. However, there is no homeomorphism $\psi\colon (\C^2,X,0)\to (\C^2,Y,0)$.
\end{example}

\subsection{Invariance of the multiplicity for real curves}\label{subsec:curves}

% \begin{example}
% For each $t\in \R$, we consider $X_t=\{(x,y)\in\R^2;\, y^3=t^3x^2\}$ and $\varphi_t: \R^2\to \R^2$ given by $\varphi_t(x,y)=(x,tx^{\frac{2}{3}}-y)$. Then, $\varphi_t$ is homeomorphism such that $\varphi(X_t)=X_0$ for all $t\in \R$. but $m(X_0)\equiv 1\, {\rm mod\,}\,2$ and $m(X_t)\equiv 0\, {\rm mod\,}\,2$ for all $t\in \R\setminus \{0\}$.
% \end{example}
We start this Subsection stating that the multiplicity ${\rm mod\,}\,2$ is not a topological invariant even in the case of real analytic curves,  even in a topologically trivial family of real analytic curves, as we can see in the next example.
\begin{example}
For each $t\in \R$, we consider $X_t=\{(x,y)\in\R^2;\, t^2y^2=x^3\}$ and $\varphi_t: \R^2\to \R^2$ given by $\varphi_t(x,y)=((x^3-t^2y^2)^{\frac{1}{3}},y)$. Then, $\varphi_t$ is homeomorphism such that $\varphi_t^{-1}(x,y)=((x^3+t^2y^2)^{\frac{1}{3}},y)$ and $\varphi_t(X_t)=X_0$ for all $t\in \R$. However, $m(X_0)\equiv 1\, {\rm mod\,}\,2$ and $m(X_t)\equiv 0\, {\rm mod\,}\,2$ for all $t\in \R\setminus \{0\}$. 
\end{example}
In the reality, if $\varphi_t$ is given by as in the above example and $f_t:\R^2\to \R$ is given by $f_t(x,y)=x^3-t^2y^2$, we have $f_t=f_0\circ \varphi_t$ for all $t\in \R$.

\begin{proposition}\label{curve}
The multiplicity ${\rm mod\,}\,2$ is an invariant for real analytic curves blow-spherical homeomorphic.
\end{proposition}
\begin{proof}
Let $X, Y\subset \R^n$ be two real analytic curves. Suppose that $X$ and $Y$ are blow-spherical homeomorphic. By Theorem \ref{multiplicities}, $C_X'$ and $C_Y'$ are homeomorphic.

We have that  
$$\# C_X'=\#({\pi'}^{-1}(1)\cap C_X')+\#({\pi'}^{-1}(-1)\cap C_X')$$
and by Proposition \ref{multcone}, we obtain that
$$
m(C_X')\equiv \#({\pi'}^{-1}(1)\cap C_X')\equiv \#({\pi'}^{-1}(-1)\cap C_X')\equiv m(X)\,{\rm mod\,} 2.
$$
Then $m(X)\equiv \frac{1}{2}\# C_X'\,{\rm mod\,} 2$. Similarly, we obtain that $m(Y)\equiv \frac{1}{2}\# C_Y'\,{\rm mod\,} 2.$
Since $C_X'$ and $C_Y'$ are homeomorphic, we have that $m(X)\equiv m(Y)\,{\rm mod\,} 2$.
\end{proof}

% Compare Proposition \ref{curve} with the result proved by J.-J. Risler (\cite{Risler:2001}, Lemma 4.5).

\subsection{Invariance of the multiplicity for analytic foliations in the plane}\label{subsec:foliations}
As a consequence of Proposition \ref{curve}, we obtain a result of invariance of the multiplicity for analytic foliations in $\R^2$ (in order to know more about the notation and definitions of the next result, see \cite{Risler:2001}).

Let $Z$ and $Z'$ be real analytic germs of plane vector fields at $0\in \mathbb{R}^2$ inducing germs of foliations $\mathcal F$ and $\mathcal F'$. Assume that in the resolution process of $Z$ (resp. $Z'$) there is no real dicritical component and that $Z$ and $Z'$ are real generalized curves. We denote by $\nu$ (resp. $\nu'$) to be the order of $Z$ (resp. $Z'$) at $0$. 

\begin{corollary}\label{risler_result}
Let $h\colon (\mathbb{R}^2,0)\to (\mathbb{R}^2,0)$ be a homeomorphism which is a topological equivalence between $\mathcal F$ and $\mathcal F'$. If $h$ is a blow-spherical homeomorphism then $\nu\equiv \nu'  \,{\rm mod}\, 2$.
\end{corollary}
\begin{proof}
Let $Z_{\C}$ (resp. $Z_{\C}'$) be the complexification of $Z$ (resp. $Z'$) and let $S_{\C}$ (resp. $S_{\C}'$) be the separatrix of $Z_{\C}$ (resp. $Z_{\C}'$).
As it was remarked in the proof of Proposition 4.4 in \cite{Risler:2001}, the real separatrix is the union of the real components of the complex separatrix. Then, 
$$
m(S)\equiv m(S_{\C},0)\, {\rm mod}\, 2
$$ 
and 
$$
m(S')\equiv m(S_{\C}',0)\, {\rm mod}\, 2,
$$ 
where $S$ (resp. $S'$) is the separatrix of $Z$ (resp. $Z'$). Moreover, by remark after Lemma 3.4 in \cite{Risler:2001}, we have
$$
 m(S_{\C},0)\equiv v +1 \,{\rm mod}\, 2
$$
and
$$
 m(S_{\C}',0)\equiv v' +1 \,{\rm mod}\, 2.
$$
By Proposition \ref{curve}, $m(S)\equiv m(S')\,{\rm mod}\, 2$, since the hypotheses imply that $S$ and $S'$ are blow-spherical homeomorphic.
Therefore, $\nu\equiv \nu'  \,{\rm mod}\, 2$.
\end{proof}

Thus, we obtain also a real version of Theorem 1.1 in \cite{Rosas:2009}.
\begin{corollary}\label{rudy_result}
Let $h\colon (\mathbb{R}^2,0)\to (\mathbb{R}^2,0)$ be a homeomorphism which is a topological equivalence between $\mathcal F$ and $\mathcal F'$. If $h$ has a derivative at the origin and $Dh_0\colon\mathbb{R}^2\to \mathbb{R}^2$ is an isomorphism, then $\nu\equiv \nu'  \,{\rm mod}\, 2$.
\end{corollary}

\subsection{Invariance of the multiplicity by blow-isomorphism for real surfaces: The embedded case}\label{subsec:surfaces}

In this Subsection, we show that the multiplicity (${\rm mod}\, 2$) of surfaces in $\mathbb{R}^3$ is invariant by embedded blow-spherical homeomorphisms.

\begin{theorem}\label{surface}
Let $X, Y\subset \R^3$ be two real analytic surfaces. If there exists a blow-spherical homeomorphism $\varphi\colon (\R^3,X,0)\to (\R^3,Y,0)$, then $m(X)\equiv m(Y)\,{\rm mod\,}\,2$.
\end{theorem}
\begin{proof}%[Proof of Theorem \ref{surface}]
By Theorem \ref{multiplicities}, we have a homeomorphism $\psi\colon (\mathbb{S}^2,C_X')\to (\mathbb{S}^2,C_{Y}')$. 

Given $\lambda,\mu\in \mathbb{S}^2\setminus C_X'$, let $\gamma \colon [0, 1] \to \mathbb{S}^n$ be an allowed path for $C_X'$ connecting $\lambda$ and $\mu$ such that $lg_{C_{Y}'}(\gamma)=d_{C_X'}(\lambda;\mu)$. Let $\beta=\psi\circ \gamma$.
Thus, the set $I_{\beta} := \{t \in [0, 1];\, \beta(t) \in C_{Y}'\}$ is finite and for every $t\in I_{\beta},$ the point $\beta(t)$ is a $C^0$ regular point of $C_{Y}'$ at which the mapping $\beta$ is topological transverse to $C_{Y}'$ (i.e., for each point $t\in I_{\beta}$, $p=\beta(t)$ is a $C^0$ regular point of $C$ and there exist open subsets $U\subset \mathbb{S}^2$ and $V\subset \R^2$ such that $(p,0)\in U\times V$ and there exists a homeomorphism $\varphi\colon U\to V$ satisfying $\varphi(p)=0$, $\varphi(U\cap C_{Y}')=V\cap \{(x_1,x_2)\in \R^2;x_2=0\}$ and $\varphi\circ \beta(s)=(0,s)$ in $(t-\delta,t+\delta)$ for some $\delta>0$).

Since ${\rm Reg}_{1}(C_{Y}')$ is dense in $C_{Y}'$, we can find a subanalytic path $\alpha \colon [0, 1] \to \mathbb{S}^2$ connecting $\tilde\lambda=\psi(\lambda)$ and $\tilde\mu=\psi(\mu)$, which is an allowed path for $C_{Y}'$ and $lg_{C_{Y}'}(\alpha)=\# I_{\beta}$ (see Figures \ref{fig1} and \ref{fig2} below).
\begin{figure}[h]
\begin{tabular}{cc}
\begin{minipage}[c][8cm][c]{6cm}
%   \begin{figure*}
    \centering \includegraphics[scale=0.3]{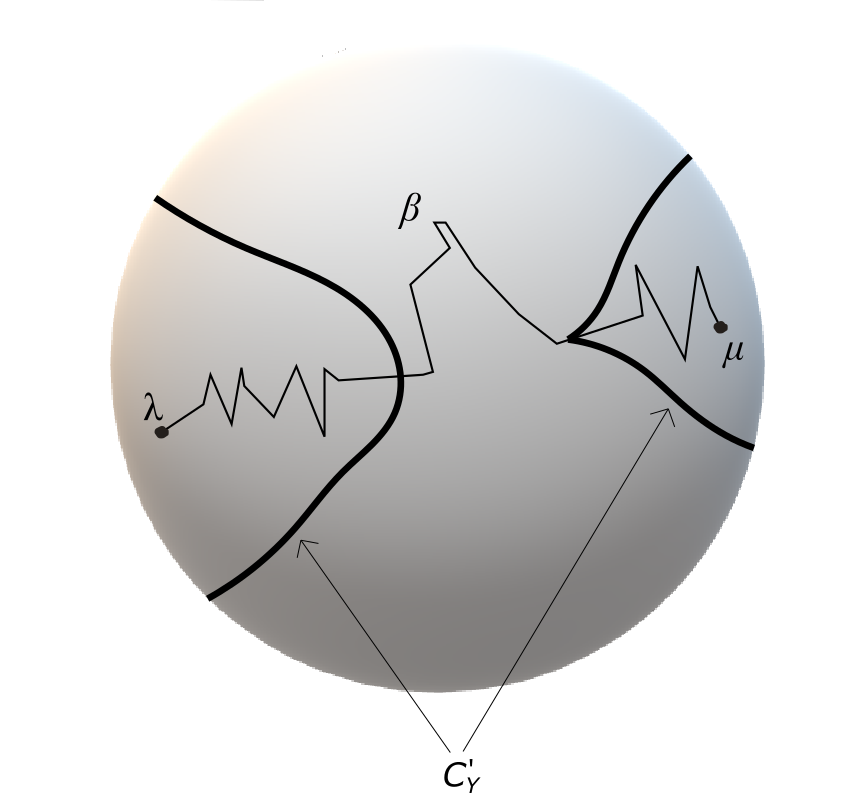}
    \caption{Almost allowed path}\label{fig1}
%   \end{figure*}
  \end{minipage} 
  & 
  \begin{minipage}[c][8cm][c]{6cm}
%   \begin{figure*}
    \centering \includegraphics[scale=0.3]{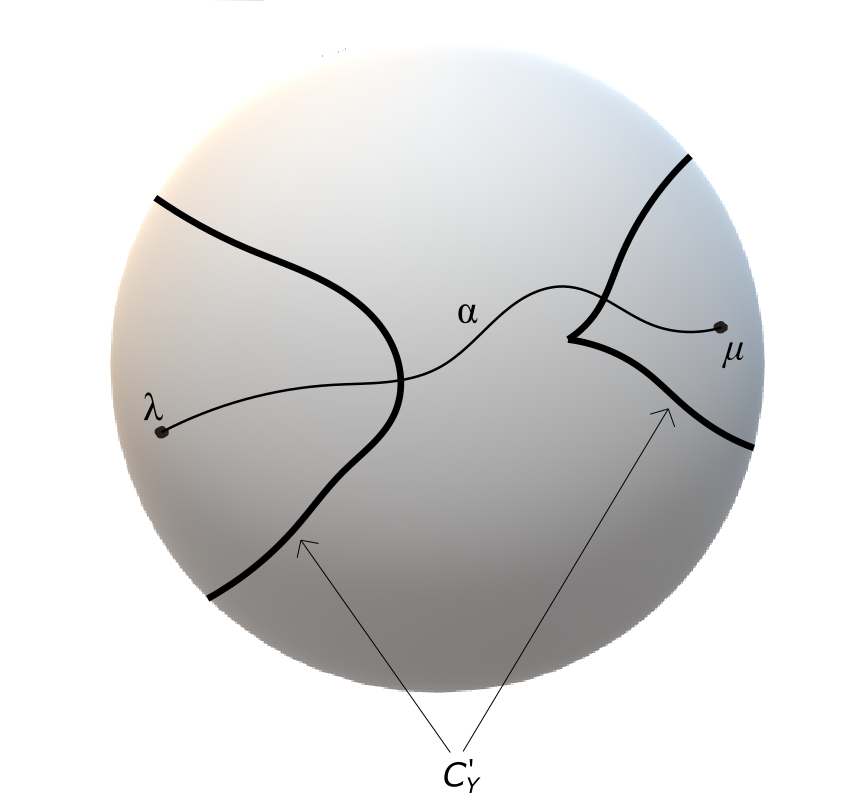}
    \caption{Allowed path}\label{fig2}
%   \end{figure*}
  \end{minipage}
\end{tabular}
\end{figure}

% Indeed, let $t_1<t_2<...<t_k$ be the numbers in $(0,1)$ such that $\beta^{-1}(I_{\beta})=\{t_1,t_2,...,t_k\}$, where $k=\# I_{\beta} $. For each $j\in\{1,...,k\}$ let $z_j=\beta(t_j)$.
Thus, we obtain that $d_{C_X'}(\lambda;\mu)\geq d_{C_{Y}'}(\psi(\lambda);\psi(\mu))$. Similarly, we obtain also $d_{C_X'}(\lambda;\mu)\leq d_{C_{Y}'}(\psi(\lambda);\psi(\mu))$. Therefore, $\delta_{C_X'}=\delta_{C_{Y}'}$ and by Lemma \ref{diameter_mult}, $m(X)\equiv m(Y)\,{\rm mod\,}\,2$.
\end{proof}

\begin{remark}\label{rem:gen_valette_one}
In Theorem 5.1 in \cite{Valette:2010}, G. Valette showed that if $X, Y\subset \R^3$ are two real analytic surfaces and $\varphi\colon (\R^3,0)\to (\R^3,0)$ is a subanalytic outer bi-Lipschitz homeomorphism such that $\varphi(X)=Y$ then $m(X)\equiv m(Y)\,{\rm mod\,}\,2$. Theorem \ref{surface} is more general than Theorem 5.1 in \cite{Valette:2010}, since any subanalytic outer bi-Lipschitz homeomorphism is a blow-spherical homeomorphism (see Proposition 3.8 in \cite{Sampaio:2020}) and by Examples \ref{blow-iso_no_bi-Lip} and \ref{blow-iso_no_bi-Lip_two}, we know that $X_{\frac{5}{2}}$ and $X_{\frac{7}{2}}$ are not outer bi-Lipschitz homeomorphic, but they have the same multiplicity and there exists a blow-spherical homeomorphism $\varphi\colon (\R^3,0)\to (\R^3,0)$ such that $\varphi(X_{\frac{5}{2}})=X_{\frac{7}{2}}$. 
\end{remark}

\subsection{Invariance of the multiplicity by blow-isomorphism for real surfaces: The non-embedded case}\label{subsec:surfaces-non-embedded}

In this Subsection, we show that the multiplicity (${\rm mod}\, 2$) of surfaces in $\mathbb{R}^3$ is invariant by blow-spherical homeomorphisms which are not necessary embedded, but with an additional hypothesis.

\begin{theorem}\label{surface-non-embedded}
Let $X, \tilde X\subset \R^3$ be two real analytic surfaces. Assume that there exists a blow-spherical homeomorphism $\varphi\colon (X,0)\to (\tilde X,0)$ such that $\nu_{\varphi}(-v)=-v$ whenever $v\in C_X'$. Then $nac(X)=nac(\tilde X)$ and $m(X)\equiv m(\tilde X)\,{\rm mod\,}\,2$.
\end{theorem}
\begin{proof}%[Proof of Theorem \ref{surface}]
It follows from Theorem \ref{multiplicities} that $C_X'\not=\emptyset$ if and only if $C_{\tilde X}'\not=\emptyset$. 
By hypothesis $\psi=\nu_{\varphi}\colon C_X'\to C_{\tilde X}'$ is a homeomorphism such that $\nu_{\varphi}(-v)=-v$ for all $v\in C_X'$.
Thus, there is no embedding $e\colon\mathbb{S}^1\to C_X'$ if and only if there is no embedding $\tilde e\colon\mathbb{S}^1\to C_{\tilde X}'$. 
% Since $0\leq nac(X)$ and $0\leq nac(\tilde X)$, if there is no embedding $e\colon\mathbb{S}^1\to C_X'$, then $nac(X)=nac(\tilde X)=0$.

% \begin{claim}\label{no_embedding}
%  If there is no embedding $e\colon\mathbb{S}^1\to C_X'$ then $nac(X)=nac(\tilde X)=0$ and $m(X)\equiv m(\tilde X)\,{\rm mod\,}\,2$.
% \end{claim}
% \begin{proof}[Proof of Claim \ref{no_embedding}]
% 
% \end{proof}
Therefore, we may assume that there are embeddings $e\colon\mathbb{S}^1\to C_X'$ and $\tilde e\colon\mathbb{S}^1\to C_{\tilde X}'$.

\begin{claim}\label{exists_allowed}
There is an allowed set $S$ of $C_X'$.
\end{claim}
\begin{proof}[Proof of Claim \ref{exists_allowed}]
We are assuming that there exists a set $S_1\subset C_X'$ which is homeomorphic to $\mathbb{S}^1$. Let $S_2=a(S_1)$.
If $S_2=S_1$ or $S_1\cap S_2 =\emptyset$,  then $S=S_1\cup S_2$ is an allowed set of $C_X'$. Thus, let us suppose that $S_1\cap S_2 \not =\emptyset$ and $S_2\not=S_1$ and fix $x_0\in S_1\cap S_2$. In particular, $-x_0\in S_1\cap S_2$. For each $i=1,2$ we consider a parametrization $e_i\colon [0,1]\to S_i$ such that $e_i(0)=e_i(1)=x_0$. Then, there exist open intervals $I_j=(a_j,b_j)$ (resp. $J_j=(c_j,d_j)$) such that 
$e_1^{-1}(S_1\setminus (S_1\cap S_2))=\bigsqcup \limits_{j=1}^k I_j$ (resp. $e_2^{-1}(S_2\setminus (S_1\cap S_2))=\bigsqcup \limits_{j=1}^k J_j$), where $ b_j\leq \inf a_{j+1}$ (resp. $\sup d_j\leq \inf c_{j+1}$) for all $j=1,\cdots, k-1$. By changing the orientation of $e_2$, if necessary, we can suppose that $e_1(a_j)=e_2(c_j)$ and $e_1(b_j)=e_2(d_j)$ for all $j=1,\cdots, k$. For each $j$, we consider a positive oriented homeomorphism $\phi_j\colon [a_j,b_j]\to [c_j,d_j]$.
We define $\tilde e_{1}, \tilde e_{2}\colon [0,1] \to C_X'$ in the following way
$$
\tilde e_{1}(t)=\left\{\begin{array}{ll}
                          e_{1}(t),& \mbox{ if } t\not\in I_j, \,\,\forall j;\\
                          e_{1}(t),& \mbox{ if } t\in [a_j,b_j] \mbox{ and } a(e_1(I_j))\subset e_{1}([0, 1]);\\
                          e_{2}\circ \phi_j(t),& \mbox{ if } t\in [c_j,d_j] \mbox{ and } a(e_1(I_j))\not\subset e_{1}([0, 1])
                         \end{array}\right.
$$
and
$$
\tilde e_{2}(t)=\left\{\begin{array}{ll}
                          e_{2}(t),& \mbox{ if } t\not\in J_j, \,\,\forall j;\\
                          e_{2}(t),& \mbox{ if } t\in [c_j,d_j] \mbox{ and } a(e_2(J_j))\subset e_{2}([0, 1]);\\
                          e_{1}\circ \phi_j^{-1}(t),& \mbox{ if } t\in [c_j,d_j] \mbox{ and } a(e_2(J_j))\not\subset e_{2}([0, 1]).
                         \end{array}\right.
$$
Now, we note that $\tilde S_1=\tilde e_1([0,1])$ and $\tilde S_2=\tilde e_2([0,1])$ are allowed sets of $C_X'$.
\end{proof}

\begin{claim}\label{decomposition}
There is a decomposition $C:=C_X'=L\cup A$ satisfying the following:
\begin{enumerate}
 \item $A$ and $L$ are $a$-invariant closed subanalytic sets and $L\cap A$ is a finite set; 
 \item $A$ is an allowed set with $nac(A)=nac(C)$;
 \item $\mathbb{S}^2\setminus L$ is path connected;
\end{enumerate}
\end{claim}
\begin{proof}[Proof of Claim \ref{decomposition}]
Since $nac(C)<+\infty$, there is a subset $A\subset C$ such that it is a maximal allowed set, let us write $E(A)$ as $E(A)=\{e_1,\cdots, e_{2k},$  $e_{2k+1},\cdots, e_{2k+m}\}$ such that $nac(C)=2k+m$, $a(Im(e_{k+j}))=Im(e_j)$ for $j\in\{1,\cdots,k\}$ and $a(Im(e_r))=Im(e_r)$ for $r\in\{2k+1,\cdots,2k+m\}$. Since $C$ and $A$ are $a$-invariant subanalytic sets and $A$ is maximal, there is no embedding $e\colon \mathbb{S}^1\to \overline{C\setminus A}$. Then, we define $L:=\overline{C\setminus A}$. Now, it is easy to verify that $L$ and $A$ satisfy the claim.
\end{proof}

\begin{claim}\label{tangent_allowed}
$A=C_X'$.
\end{claim}
\begin{proof}[Proof of Claim \ref{tangent_allowed}]
Suppose that $L\not\subset A$. Since $C_X'$ has no isolated point, there exists a subset $I\subset L\setminus A$ that is homeomorphic to $(0,1)$. 
Let $\lambda\in\mathbb{S}^2\setminus C_X'$ and let $\gamma : [0, 1] \to \mathbb{S}^3$ be an allowed path for $C_X'$ such that $\gamma(0)=\lambda$ and $\gamma(1)=-\lambda$. Since $\mathbb{S}^2\setminus L$ is path connected, we can assume that $\Gamma\subset \mathbb{S}^2\setminus (L\setminus I)$ and $\gamma$ meets transversally $I$ at exactly one point, where $\Gamma=\gamma ([0, 1])$.
Let $\beta : [0, 1] \to \mathbb{S}^n$ be an allowed path for $C_X'$ such that $\beta(0)=\lambda$, $\beta(1)=-\lambda$ and $\beta ([0, 1])\subset \mathbb{S}^2\setminus L$. By Lemma \ref{euler_cycle}, $A$ is an Euler cycle and, then, $\gamma$ and $\beta$ are also allowed paths for $A$ and by Lemma \ref{path_independence}, 
$$\tilde d_{A}(-\lambda,\lambda)\equiv lg_{A}(\gamma) \equiv lg_{A}(\beta) \,{\rm mod\,} 2$$
and 
$$\tilde d_{C_X'}(-\lambda,\lambda)\equiv lg_{C_X'}(\gamma) \equiv lg_{C_X'}(\beta) \,{\rm mod\,} 2.$$
However, by the choice of $\gamma$ and $\beta$, we have $lg_{C_X'}(\gamma)\equiv lg_{A}(\gamma) +1\,{\rm mod\,} 2$ and  $lg_{C_X'}(\beta) \equiv lg_{A}(\beta) \,{\rm mod\,} 2$, which is a contradiction. Therefore $L\subset A$ and this finish the proof of Claim \ref{tangent_allowed}.
\end{proof}
Similarly, $\tilde A=C_{\tilde X}'$ is a maximal allowed set of $C_{\tilde X}'$.
\begin{claim}\label{nac_equal}
$nac(X)= nac(\tilde X)$.
\end{claim}
\begin{proof}[Proof of Claim \ref{nac_equal}]
We write $C_X'=\bigcup\limits_{r=1}^{2k+m} A_r$ (resp. $C_{\tilde X}'=\bigcup\limits_{r=1}^{2\tilde k+\tilde m} \tilde A_r$) such that each $A_i$ (resp. $\tilde A_i$) is homeomorphic to $\mathbb{S}^1$, $A_r\cap A_{r'}$ (resp. $\tilde A_r\cap \tilde A_{r'}$) is a finite set whenever $r\not =r'$ (resp. $\tilde r\not =\tilde r'$), $a(A_r)=A_{k+r}$ whenever $r\in\{1,\cdots,k\}$ (resp. $a(\tilde A_{\tilde r})=\tilde A_{\tilde k+\tilde r}$ whenever $\tilde r\in\{1,\cdots,\tilde k\}$), $a(A_r)=A_{r}$ whenever $r\in\{2k+1,\cdots,2k+m\}$ (resp. $a(\tilde A_{\tilde r})=\tilde A_{\tilde r}$ whenever $\tilde r\in\{2\tilde k+1,\cdots,2\tilde k+m\}$) and $nac(X)=2k+m$ (resp. $nac(\tilde X)=2\tilde k+\tilde m$). 

Since $a\circ \psi=\psi\circ a$, $\bigcup\limits_{r=1}^{2k+m} \psi(A_r)$ is an allowed set of $C_{\tilde X}'$, which implies that $nac(X)\leq nac(\tilde X)$. By using $\psi^{-1}$ instead of $\psi$, we obtain $nac(\tilde X)\leq nac(X)$. Therefore, $nac(X)= nac(\tilde X)$.
\end{proof}
Let $\lambda\in\mathbb{S}^2\setminus C_X'$ and let $\gamma : [0, 1] \to \mathbb{S}^2$ be an allowed path for $C_X'$ such that $\gamma(0)=\lambda$ and $\gamma(1)=-\lambda$. Since $\mathbb{S}^2\setminus L$ is path connected, we can assume that $\gamma ([0, 1])\subset \mathbb{S}^2\setminus L$. By Lemma \ref{euler_cycle}, $A$ is an Euler cycle and, then, by Lemma \ref{mult_dist}, $\tilde d_{A}(-\lambda,\lambda)\equiv \tilde d_{C_X'}(-\lambda,\lambda)\equiv m(X)\,{\rm mod\,} 2$ and $\gamma$ is an allowed path for $A$. 
Then, by Lemma \ref{mult_allowed}, we have
$$
\tilde d_{A}(-\lambda,\lambda)\equiv \sum\limits_{i=1}^{2k+m}\tilde d_{A_i}(-\lambda,\lambda)\equiv m\, \,{\rm mod\,} 2.
$$
Thus, $nac(X)\equiv m(X)\,{\rm mod\,} 2$, since $nac(X)=2k+m$. Similarly, we obtain $nac(\tilde X)\equiv m(\tilde X)\,{\rm mod\,} 2$.
Since $nac(X)=nac(\tilde X)$, we have that $m(X)\equiv m(\tilde X)\,{\rm mod\,} 2$.
\end{proof}

A consequence of above result is presented in the next section (see Corollary \ref{appl_teo_surface}).

Next example shows that Theorem \ref{surface-non-embedded} is not a direct consequence of the results in \cite{Valette:2010}.
\begin{example}
Let $X=\{x,y,z)\in\mathbb{R}^3; (x^2+y^2-z^2)(x^2+z^2-\frac{y^2}{4})=0\}$ and $Y=\{x,y,z)\in\mathbb{R}^3; (x^2+y^2-z^2)(x^2+y^2-\frac{z^2}{4})=0\}$. The mapping $\psi\colon X\to Y$ given by
$$
\psi(x,y,z)=\left\{\begin{array}{ll}
(x,y,z),&\mbox{ if } x^2+y^2-z^2=0\\
(x,z,y),&\mbox{ if } x^2+z^2-\frac{y^2}{4}=0
\end{array}\right.
$$
is a bow-spherical homeomorphism such that $\nu_{\psi}(-v)=-\nu_{\psi}(v)$ for all $v\in X$. However, there is no bi-Lipschitz homeomorphism $\varphi\colon (\R^3,0)\to (\R^3,0)$ such that $\varphi(X)=Y$.
\end{example}

\subsection{Invariance of the multiplicity by arc-analytic blow-isomorphism}\label{subsec:arc-analytic}

\begin{definition}
Let $C\subset \mathbb{S}^n$ be an Euler cycle. 
A path $\gamma \colon [0, 1] \to \mathbb{S}^n$ is said to be an {\bf almost allowed path for} $C$ if the set $I_{\gamma} = \{t \in [0, 1];\, \gamma(t) \in C\}$ is finite and for every $t\in I_{\gamma},$ the point $\gamma(t)$ is a $C^0$ regular point of $C$ at which the mapping $\gamma$ is topological transverse to $C$ (i.e., for each point $t\in I_{\gamma}$, $p=\gamma(t)$ is a $C^0$ regular point of $C$ and there exist open subsets $U\subset \mathbb{S}^n$ and $V\subset \R^n$ such that $(p,0)\in U\times V$ and there exists a homeomorphism $\varphi\colon U\to V$ satisfying $\varphi(p)=0$, $\varphi(U\cap C)=V\cap \{(x_1,...,x_n)\in \R^n;x_n=0\}$ and $\varphi\circ \gamma(s)=(0,...,0,s)$ in $(t-\delta,t+\delta)$ for some $\delta>0$).
\end{definition}
\begin{remark}
If $\gamma \colon [0, 1] \to \mathbb{S}^n$ is an almost allowed path (for $C$) then it is an allowed path (for $C$).
\end{remark}

If $\gamma$ is an almost allowed path (for $C$), we define also
$$
lg_C(\gamma)=\#I_{\gamma}.
$$

For $\lambda, \mu\in \mathbb{S}^n\setminus C$, we define
$$
\widetilde{d}_C(\lambda;\mu)=\min\{lg_C(\gamma);\, \gamma \mbox{ is an almost allowed path joining }\lambda\mbox{ and }\mu\}
$$

% 
% Thus, we have the following key remark.
\begin{remark}\label{key_remark}
Let $C\subset \mathbb{S}^n$ be an Euler cycle and let $\lambda, \mu\in \mathbb{S}^n\setminus C$. It is clear that $\widetilde{d}_C(\lambda, \mu)\leq d_C(\lambda, \mu)$. If $\gamma \colon [0, 1] \to \mathbb{S}^n$ is an almost allowed path for $C$ connecting $\lambda$ and $\mu$ such that $\widetilde{d}_C(\lambda, \mu)=lg_C(\gamma)$, since ${\rm Reg}_{1}(C)$ is dense in $C$, we can find a path $\tilde\gamma \colon [0, 1] \to \mathbb{S}^n$ connecting $\lambda$ and $\mu$, which is an allowed path for $C$ and $lg_C(\tilde\gamma)=lg_C(\gamma)$. Therefore, $\widetilde{d}_C=d_C$.
\end{remark}

% \begin{proposition}
% Let $X, Y\subset \R^{n+1}$ be real analytic subsets. Suppose $X$ and $Y$ are bi-Lipschitz homeomorphic. If $\dim C_X'<n-1$ or $C(X,0)$ is not an union of lines passing through the origin $0\in \R^{n+1}$, then $m(X)\equiv 0\, \,{\rm mod\,}\,2$ and $m(Y)\equiv 0\,\,{\rm mod\,}\, 2$.
% \end{proposition}

\begin{definition}
Let $f\colon \R^n\to \R^m$ be a function.
We say that $f$ is {\bf arc-analytic} if for any analytic curve $\alpha\colon (-1,1)\to \R^n$ we have that there exists $0<\varepsilon\leq 1$ such that $f\circ \alpha$ is analytic in $(-\varepsilon,\varepsilon)$.
We say that $f$ is {\bf image arc-analytic} if for any analytic curve $\alpha\colon (-1,1)\to \R^n$ we have that there exist $0<\varepsilon\leq 1$ and an analytic curve $\gamma \colon (-\varepsilon,\varepsilon) \to\R^n$ such that $\gamma(0)=f\circ \alpha (0)$ and $(Im(\gamma),0)=(Im(f\circ \alpha),0)$.
\end{definition}
\begin{remark}
 It is clear that if $f$ is arc-analytic then $f$ is image arc-analytic. However, the converse it is not true in general. For example, it is easy to see that $f\colon \R\to \R$ given by $f(t)=t^{\frac{1}{3}}$ is image arc-analytic, but it is not arc-analytic.
\end{remark}

Moreover, there are image arc-analytic and bi-Lipschitz mappings that are not arc-analytic, as we can see in the next example.
\begin{example}\label{non-arc-analytic}
Let $\phi\colon (\R^2,0)\to (\R^2,0)$ be the mapping given by $\phi(x,y)=(x,y+x^{\frac{4}{3}})$. It is clear that $\phi$ is bi-Lipschitz around the origin and $t\mapsto\phi(t,0)=(t,t^{\frac{4}{3}})$ is not analytic. Then $\phi$ is not arc-analytic. However, for any analytic curve $\alpha\colon (-1,1)\to \R^2$, we have that $\gamma(t):=\phi\circ \alpha (t^3)$ is analytic in $(-\varepsilon,\varepsilon)$ for some $\varepsilon>0$.
\end{example}

\begin{remark}\label{lip-blow}
Let $\varphi\colon (\R^n,0)\to (\R^n,0)$ be an arc-analytic and bi-Lipschitz homeomorphism. Then, $\varphi$ is a blow-spherical homeomorphism such that $\nu_{\varphi}(-v)=-\nu_{\varphi}(v)$ for any $v\in\mathbb{S}^{n-1}$.
\end{remark}

Next result is a generalization of Theorem 3.6 in \cite{Valette:2010}.
\begin{theorem}\label{image-arc-Lip}
Let $X$ and $Y$ be two real analytic hypersurfaces in $\R^n$. If there exists a blow-spherical homeomorphism $\varphi\colon (\R^n,X,0)\to (\R^n,Y,0)$ such that $\varphi$ is image arc-analytic, then $m(X)\equiv m(Y)\,{\rm mod\,} 2$.
\end{theorem}
\begin{proof}
Fixed $v\in \R^n\setminus\{0\}$, by the hypotheses, there exists an analytic curve $\gamma\colon (-\varepsilon,\varepsilon)\to \R^n$ such that $(Im(\gamma),0)=(Im(\beta),0)$, where $\beta(t):=\varphi(tv)$. Thus, there exists $\delta>0$ such that $\beta((-\delta,\delta))\subset \Gamma:=Im(\gamma)$. Moreover, since $\varphi$ is a homeomorphism, we can assume that $\gamma$ is injective, since $(Im(\beta),0)\not=(\{0\},0)$ and also $(Im(\beta),0)$ is not homeomorphic to $([0,1),0)$. In particular, $\Gamma$ is blow-spherical regular at $0$, then by Theorem \ref{half-line}, 
$$
w:=\lim\limits_{t\to 0^+}\frac{\gamma(t)}{\|\gamma(t)\|}=-\lim\limits_{t\to 0^-}\frac{\gamma(t)}{\|\gamma(t)\|}.
$$
Let $\Gamma_1$ and $\Gamma_2$ be the connected components of $\Gamma\setminus \{0\}$ such that $\beta((-\delta,0))\subset \Gamma_1$ and $\beta((0,\delta))\subset \Gamma_2$. We define $a=1$ and $b=0$ if $\Gamma_1=\gamma((-\varepsilon,0))$ and, $a=0$ and $b=1$ if $\Gamma_1=\gamma((0,\varepsilon))$.
Then, for each sequence $\{s_k\}\subset (0,\delta)$ such that $\lim s_k=0$ there exists a sequence $\{\tilde s_k\}\subset \Gamma_2$ such that $\lim \tilde s_k=0$ and $\gamma(\tilde s_k)=\beta(s_k)$ for $k>>1$. Then,
$$
\lim\limits_{k\to +\infty}\frac{\beta(s_k)}{\|\beta(s_k)\|}=\lim\limits_{k\to +\infty}\frac{\gamma(\tilde s_k)}{\|\gamma(\tilde s_k)\|}=(-1)^aw.
$$
Therefore $
\lim\limits_{t\to 0^+}\frac{\beta(t)}{\|\beta(t)\|}=(-1)^aw
$
and by the same reason, $
\lim\limits_{t\to 0^-}\frac{\beta(t)}{\|\beta(t)\|}=(-1)^bw.
$
In particular, 
$
\lim\limits_{t\to 0^+}\frac{\beta(\lambda t)}{\|\beta(\lambda t)\|}=(-1)^aw
$
and
$
\lim\limits_{t\to 0^-}\frac{\beta(\lambda t)}{\|\beta(\lambda t)\|}=(-1)^bw
$
for any $\lambda \in (0,+\infty)$. This implies that the homeomorphism $\nu_{\varphi}:\mathbb{S}^{n-1}\to \mathbb{S}^{n-1}$  
satisfies $\nu_{\varphi}(-x)=-\nu_{\varphi}(x)$ for any $x\in\mathbb{S}^{n-1}$. Thus, for a generic $v\in \mathbb{S}^{n-1}$, by Lemma \ref{path_independence}, we have
$$
m(X)\equiv d_{C_X'}(v;-v)=d_{C_Y'}(\nu_{\varphi}(v);-\nu_{\varphi}(v))\equiv m(Y)\,{\rm mod\,} 2.
$$
\end{proof}

In fact, we have the following.
\begin{corollary}\label{simmetric-blow}
Let $X$ and $Y$ be two real analytic hypersurfaces in $\R^n$. If $\varphi\colon (\R^n,X,0)\to (\R^n,Y,0)$ a blow-spherical homeomorphism such that $\nu_{\varphi}(-v)=-\nu_{\varphi}(v)$ for any $v\in\mathbb{S}^{n-1}$, then $m(X)\equiv m(Y)\,{\rm mod\,} 2$.
\end{corollary}
A consequence of Theorem \ref{surface-non-embedded} and the proof of Theorem \ref{image-arc-Lip} is the following:
\begin{corollary}\label{appl_teo_surface}
Let $X$ and $Y$ be two real analytic surfaces in $\R^3$. If there exists an image arc-analytic bi-Lipschitz homeomorphism $\varphi\colon (X,0)\to (Y,0)$, then $m(X)\equiv m(Y)\,{\rm mod\,} 2$.
\end{corollary}

\begin{remark}\label{rem:gen_valette_two}
 Theorem \ref{image-arc-Lip} is really a generalization of Theorem 3.6 in \cite{Valette:2010}, as we can see in the next example.
\end{remark}
\begin{example}
Let $\phi\colon (\R^2,0)\to (\R^2,0)$ be the bi-Lipschitz homeomorphism of Example \ref{non-arc-analytic}. If $X=\{(x,y)\in\R^2;\, y=0\}$ and $Y=\{(x,y)\in\R^2;\, y^3=x^4\}$ then $\phi(X)=Y$. However, by Proposition 3.2 in \cite{FukuiKP:2004}, there is no arc-analytic bi-Lipschitz homeomorphism $\varphi\colon (\R^2,0)\to (\R^2,0)$ such that $\varphi(X)=Y$.
\end{example}

\noindent {\bf Acknowledgements.} The author would like to thank Eur\'ipedes C. da Silva for his interest in this research.

\end{document}